\documentclass[11pt]{article}
\usepackage{cmap}
\usepackage[T1]{fontenc}
\usepackage{lmodern}
\usepackage{amsmath,amssymb,amsthm}
\usepackage{array}
\usepackage{booktabs}
\usepackage{graphicx}
\usepackage[margin=1.15in]{geometry}
\usepackage{microtype}
\usepackage[section]{placeins}
\usepackage[font=small,labelfont=bf,labelsep=period]{caption}
\usepackage[dvipsnames]{xcolor}
\usepackage[
  backend=biber,
  style=alphabetic,
  sorting=nyt,
  maxbibnames=99,
  maxalphanames=4,
  natbib=true
]{biblatex}
\appto{\bibsetup}{\emergencystretch=1em}
\usepackage{aliascnt}
\usepackage{hyperref}
\usepackage[nameinlink,noabbrev,capitalise]{cleveref}
\graphicspath{{figures/}}
\numberwithin{equation}{section}

\newtheorem{theorem}{Theorem}[section]
\newaliascnt{proposition}{theorem}
\newtheorem{proposition}[proposition]{Proposition}
\aliascntresetthe{proposition}
\newaliascnt{lemma}{theorem}
\newtheorem{lemma}[lemma]{Lemma}
\aliascntresetthe{lemma}
\newaliascnt{corollary}{theorem}
\newtheorem{corollary}[corollary]{Corollary}
\aliascntresetthe{corollary}
\newaliascnt{conjecture}{theorem}
\newtheorem{conjecture}[conjecture]{Conjecture}
\aliascntresetthe{conjecture}
\newtheorem{hlconjecture}{Conjecture}
\theoremstyle{remark}
\newaliascnt{remark}{theorem}
\newtheorem{remark}[remark]{Remark}
\aliascntresetthe{remark}
\crefname{theorem}{\mbox{Theorem}}{\mbox{Theorems}}
\Crefname{theorem}{\mbox{Theorem}}{\mbox{Theorems}}
\crefname{proposition}{\mbox{Proposition}}{\mbox{Propositions}}
\Crefname{proposition}{\mbox{Proposition}}{\mbox{Propositions}}
\crefname{lemma}{\mbox{Lemma}}{\mbox{Lemmas}}
\Crefname{lemma}{\mbox{Lemma}}{\mbox{Lemmas}}
\crefname{corollary}{\mbox{Corollary}}{\mbox{Corollaries}}
\Crefname{corollary}{\mbox{Corollary}}{\mbox{Corollaries}}
\crefname{conjecture}{\mbox{Conjecture}}{\mbox{Conjectures}}
\Crefname{conjecture}{\mbox{Conjecture}}{\mbox{Conjectures}}
\crefname{hlconjecture}{\mbox{Conjecture}}{\mbox{Conjectures}}
\Crefname{hlconjecture}{\mbox{Conjecture}}{\mbox{Conjectures}}
\crefname{remark}{\mbox{Remark}}{\mbox{Remarks}}
\Crefname{remark}{\mbox{Remark}}{\mbox{Remarks}}
\crefname{section}{\mbox{Section}}{\mbox{Sections}}
\Crefname{section}{\mbox{Section}}{\mbox{Sections}}
\crefname{appendix}{\mbox{Appendix}}{\mbox{Appendices}}
\Crefname{appendix}{\mbox{Appendix}}{\mbox{Appendices}}
\crefname{figure}{\mbox{Figure}}{\mbox{Figures}}
\Crefname{figure}{\mbox{Figure}}{\mbox{Figures}}
\crefname{table}{\mbox{Table}}{\mbox{Tables}}
\Crefname{table}{\mbox{Table}}{\mbox{Tables}}

\newcommand{\E}{\mathbb{E}}
\newcommand{\Prob}{\mathbb{P}}
\newcommand{\orth}{\bar\rho}
\newcommand{\Rb}{R_\partial}
\DeclareMathOperator{\BS}{BS}

\title{The Right Wing of the SABR Smile with an Absorbing Boundary:\\
Resolution of Conjectures of Henry-Labord\`ere}
\author{%
Zeyu Cao\thanks{Independent Researcher, Long Island City, NY 11101, USA.
Email: \texttt{josephcao891011@gmail.com}.}
\and
Shaosai Huang\thanks{Kspectra Research Inc., Toronto, ON, M2N 0G3, Canada.
Email: \texttt{arthur.foxie.huang@kspectra.ai}.}}
\date{November 7, 2026}
\hypersetup{
  colorlinks=true,
  linkcolor=MidnightBlue,
  citecolor=OliveGreen,
  urlcolor=BrickRed,
  pdftitle={The Right Wing of the SABR Smile with an Absorbing Boundary},
  pdfauthor={Zeyu Cao and Shaosai Huang},
  pdfsubject={Right-tail asymptotics of the absorbed SABR model and Henry-Labord\`ere's conjecture},
  pdfkeywords={SABR model, implied-volatility wings, absorbed diffusions,
    logarithmic tail asymptotics, moving-boundary crossing},
  pdflang={en-US}
}

\begin{document}
\maketitle

\begin{abstract}
We determine the fixed-maturity right tail of the SABR model with $\beta\in(0,1)$ absorbed
at zero, for every correlation $\rho$: the probability $P(k)$ that the forward ends above
$f_0e^k$ satisfies $-k^{-2}\ln P(k)\to(1-\beta)^2/(2\nu^2T(1-(\rho\wedge0)^2))$. The
implied volatility therefore tends to $\nu\sqrt{1-(\rho\wedge0)^2}/(1-\beta)$.
Henry-Labord\`ere's conjectured limit $\nu/(1-\beta)$, which Hagan's leading-order formula
also gives, holds for $\rho\ge0$ and fails for $\rho<0$, where survival keeps the noise
orthogonal to volatility above a moving square-root barrier. The Gaussian estimate behind
his second conjecture does not hold for $\rho<0$; the rate is instead given by the
intrinsic distance within the survival domain.
\end{abstract}




\section{Introduction}\label{sec:intro}

\noindent Absorption can alter a fixed-time logarithmic tail asymptotic even when the
absorbing boundary is disjoint from the terminal target set. We show that this happens
for the SABR diffusion absorbed at zero: the volatility excursion that produces a large
terminal value can simultaneously drive the forward toward the boundary, so paths
contributing to the right tail incur an additional Brownian boundary-crossing cost. The
main theorem concerns the terminal right tail of the absorbed diffusion; the
implied-volatility wing then follows from a model-free transfer principle.

\subsection{The Problem and Henry-Labord\`ere's Conjectures}\label{sec:problem}
The SABR model of Hagan, Kumar, Lesniewski and Woodward \cite{HKLW} is a widely used
stochastic-volatility parametrization of implied-volatility smiles in interest-rate
markets and beyond:
\begin{equation}\label{eq:sabr}
dF_t = a_t F_t^{\beta}\, dW_t, \qquad da_t = \nu a_t\, dZ_t, \qquad d\langle W,Z\rangle_t=\rho\,dt,
\end{equation}
with $F_0=f_0>0$, $a_0=\alpha>0$, $\beta\in(0,1)$, $\nu>0$ and $\rho\in(-1,1)$.
The forward is absorbed at zero: it is stopped at its first zero $\zeta$, a boundary
condition that is part of the model (\cref{sec:setup} gives the construction). Throughout, $T>0$ is a \emph{fixed}
maturity, $k=\ln(K/f_0)$ is the log-moneyness, and $\sigma_{\BS}(k)$ denotes the
Black--Scholes implied volatility of the undiscounted call with strike $K$ and maturity
$T$.

The right wing of the smile is used wherever prices far above the forward are needed:
quoted smiles are extrapolated beyond the last liquid strike, and payoffs linked to
constant-maturity swap rates are priced by static replication, which integrates swaption
prices over all strikes \cite{Hagan03}. In these uses the model's wing is often evaluated
through Hagan's asymptotic formula \cite{HKLW,Obloj}, so the large-strike behaviour of the
model itself is a practical question as well as a mathematical one.

At fixed maturity, the extreme right wing probes joint rare excursions of volatility and
the forward. The unrestricted hyperbolic geometry of the model suggests a
correlation-independent leading coefficient, whereas the right-tail event for the
absorbed process requires survival in $(0,\infty)$. The question is whether this survival
constraint is negligible on the leading $k^2$ scale.

In Chapter~10 of his monograph \cite{HL}, Henry-Labord\`ere writes the generator of a
stochastic-volatility model as a Schr\"odinger-type operator on a Riemannian manifold: a
Laplace--Beltrami part determined by the diffusion matrix, a connection one-form, and a
scalar potential~$Q$. Assuming a Gaussian bound of Li--Yau--Davies type
\cite{LiYau,Davies} for the associated heat kernel (his (10.21)), he obtains a wing
formula, his (10.22), in which the extreme strikes are governed by the geodesic distance,
in the \emph{unrestricted} metric, from the initial point to the terminal fibre
$\{F=f\}$. For SABR this metric is hyperbolic, and the resulting limit is
$\nu/(1-\beta)$. He also evaluates the scalar potential at the saddle point $a_{\min}(f)$ of
the fibre integral, the point of the fibre $\{F=f\}$ nearest to the initial point, and
finds that it stays bounded as $f\to\infty$; its limit is \eqref{eq:Q} in
\cref{sec:potential}. On the strength of this boundedness, Henry-Labord\`ere poses two
conjectures, both on p.~306 of \cite{HL}, which we restate in our notation, keeping his
numbering.

\begin{hlconjecture}[{\cite[p.~306]{HL}}]\label{conj:HL1}
For SABR with $\beta<1$, $\sigma_{\BS}(k)\to \nu/(1-\beta)$ as $k\to\infty$, for every
$\rho\in(-1,1)$.
\end{hlconjecture}

The second conjecture concerns the Gaussian estimate itself.

\begin{hlconjecture}[{\cite[p.~306]{HL}}]\label{conj:HL2}
If the potential $Q(a_{\min})$ is bounded, the Gaussian wing estimate \textup{(10.22)} of
\cite{HL} is valid.
\end{hlconjecture}

\Cref{conj:HL1} is the proposed right-wing limit; \cref{conj:HL2} would justify the
unrestricted-distance Gaussian estimate from the boundedness of the potential at the
saddle point. At leading logarithmic order, and at the level of tail probabilities, that
estimate reads
\begin{equation}\label{eq:HL22}
-\ln\Prob\bigl(F_T\ge f_0e^k\bigr)=\frac{d_{\mathrm{free}}(z_0,\mathcal T_k)^2}{2T}+o(k^2),
\end{equation}
where $d_{\mathrm{free}}$ is the distance in the unrestricted hyperbolic metric from the
initial state $z_0$ to the terminal target set $\mathcal T_k$, both defined in
\cref{sec:geometry}. Since $d_{\mathrm{free}}(z_0,\mathcal T_k)=(1-\beta)k/\nu+O(1)$ for every
$\rho$ (\cref{prop:killed-distance}), \eqref{eq:HL22} predicts the limit of
\cref{conj:HL1}. The book states \cref{conj:HL1} without restriction on $\rho$ and
immediately attributes the case $\rho=0$ to Benaim, Friz and Lee \cite{BFL} (its
reference~[55]), which fixes the intended generality. To our knowledge, the fixed-maturity correlated case $\rho\ne0$,
$\beta\in(0,1)$ has remained open (see \cref{sec:lit}). The difficulty is the gap
between evaluating a coefficient at the unrestricted saddle point and controlling the
Dirichlet semigroup of the diffusion killed at the boundary of its state space.

\subsection{Main Results}\label{sec:results}
We measure the right tail by the probability that the terminal forward ends at or above a
large strike, and compare its decay with two quadratic rates:
\begin{equation}\label{eq:PK}
P(k) := \Prob\bigl(F_T\ge f_0 e^{k}\bigr), \qquad
c:=\frac{(1-\beta)^2}{2\nu^2T}, \qquad c_-:=\frac{c}{1-\rho^2}.
\end{equation}
For $\rho\ne0$ the rate $c_-$ exceeds $c$, and our main result shows which of the two
governs the tail.

\begin{theorem}[Logarithmic Right-Tail Asymptotics for Absorbed SABR]\label{thm:main-tail}
For every ${\rho\in(-1,1)}$,
\[
\lim_{k\to\infty}\frac{-\ln P(k)}{k^2}
=
\begin{cases}
c_-=\dfrac{(1-\beta)^2}{2\nu^2T(1-\rho^2)},&\rho<0,\\[6pt]
c=\dfrac{(1-\beta)^2}{2\nu^2T},&\rho\ge0.
\end{cases}
\]
\end{theorem}

In both cases the limit is $c/s_\rho^2$, where $s_\rho:=\sqrt{1-(\rho\wedge0)^2}$ equals
$\sqrt{1-\rho^2}$ for $\rho<0$ and $1$ for $\rho\ge0$; we use this notation throughout.
Through a model-free transfer from tail probabilities to call prices
(\cref{lem:transfer}), the decay rate determines the right wing of the
implied-volatility smile.

\begin{corollary}[Right-Wing Implied-Volatility Limit]\label{cor:wing}
For every $\rho\in(-1,1)$,
\[
\lim_{k\to\infty}\sigma_{\BS}(k)
=
\begin{cases}
\dfrac{\nu\sqrt{1-\rho^2}}{1-\beta},&\rho<0,\\[6pt]
\dfrac{\nu}{1-\beta},&\rho\ge0.
\end{cases}
\]
\end{corollary}

Comparing this limit with Henry-Labord\`ere's prediction settles \cref{conj:HL1}.

\begin{corollary}[Resolution of \cref{conj:HL1}]\label{cor:resolution}
For every $\beta\in(0,1)$, the conjecture holds for $\rho\in[0,1)$ and fails for every
$\rho\in(-1,0)$. The wing limit
\begin{equation*}
\sigma_\infty(\rho)=\frac{\nu s_\rho}{1-\beta},
\end{equation*}
as a function of $\rho$, is $C^1$ but not $C^2$ at $\rho=0$: its second derivative from
the left is $-\nu/(1-\beta)$ and from the right is $0$. It does not depend on $\alpha$,
$f_0$ or $T$, and depends on $\beta$ only through the factor $(1-\beta)^{-1}$.
\end{corollary}

Both corollaries follow from \cref{thm:main-tail} and the integrability
$\E F_T\le f_0$ alone.
\cref{fig:wing} shows the limit as a function of $\rho$, together with the value
predicted by \cref{conj:HL1}.

\begin{figure}[t]
\centering
\includegraphics[width=0.74\textwidth]{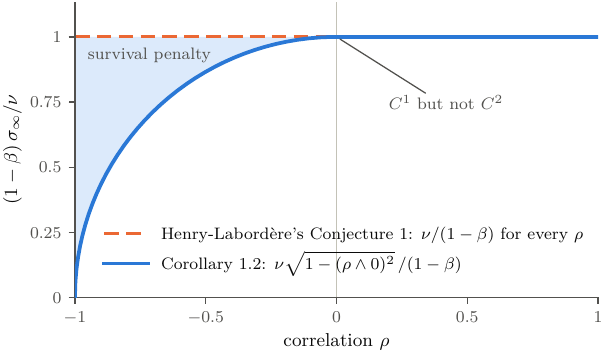}
\caption{The right-wing limit $\sigma_\infty$ of \cref{cor:wing}, normalized by
$\nu/(1-\beta)$, against the correlation $\rho$. Henry-Labord\`ere's \cref{conj:HL1} predicts the
value~$1$
for every $\rho$. It holds for $\rho\ge0$; for $\rho<0$ survival lowers the limit to
$\sqrt{1-\rho^2}$ (shaded gap).}
\label{fig:wing}
\end{figure}

The same limit separates the model from the formula most often used to evaluate it.
Hagan's expansion \cite{HKLW}, with Ob{\l}{\'o}j's correction of its leading term
\cite{Obloj}, has a leading term that tends to $\nu/(1-\beta)$ as $k\to\infty$ for every
$\rho$; with its first-order time correction, the formula tends to
\begin{equation*}
\frac{\nu}{1-\beta}\Bigl(1+\frac{(2-3\rho^2)\nu^2T}{24}\Bigr).
\end{equation*}
For $\rho<0$ this exceeds the model's limit by the factor
$(1+(2-3\rho^2)\nu^2T/24)/\sqrt{1-\rho^2}$. \Cref{sec:num} shows that the gap is already
visible at moderate strikes (\cref{fig:hagan}). It matters wherever the far wing is
used, in smile extrapolation and in replication over all strikes.

Our third result explains the correction geometrically. In \cref{sec:geometry} we
write the model in Fermi coordinates $(R,q)$ of the hyperbolic plane, in which the
survival domain is the coordinate half-plane $D=\{R>\Rb\}$ (\cref{fig:geometry}), and
identify the decay rate with the intrinsic Riemannian distance $d_D$ within $D$: for every $\rho\in(-1,1)$,
\begin{equation}\label{eq:intro-c2}
-\ln P(k)=\frac{d_D(z_0,\mathcal T_k)^2}{2T}+o(k^2)
\end{equation}
(\cref{cor:corrected-c2}), where $z_0$ is the initial state and $\mathcal T_k$
the terminal target set. The unrestricted distance yields the rate $c$ for every $\rho$;
the restriction to $D$ changes the leading coefficient exactly when $\rho<0$. Read
geometrically, the intrinsic distance takes the place of the unrestricted one;
\eqref{eq:intro-c2} records this for SABR, as a consequence of \cref{thm:main-tail} and an
explicit distance computation, not as a general principle. For \cref{conj:HL2} itself it
gives the following.

\begin{corollary}[{\Cref{conj:HL2} for Negative Correlation}]\label{cor:conj2}
For $\rho\in(-1,0)$ the hypothesis of \cref{conj:HL2} holds, since the limit
\eqref{eq:Q} is finite, but the leading-order estimate \eqref{eq:HL22} does not hold:
$-\ln P(k)/k^2\to c_-$, whereas $d_{\mathrm{free}}(z_0,\mathcal T_k)^2/(2Tk^2)\to c<c_-$.
\end{corollary}
\begin{proof}
The limit \eqref{eq:Q} is computed in \cref{sec:potential}; the two rates are those of
\cref{thm:main-tail} and \cref{prop:killed-distance}.
\end{proof}

\subsection{Mechanism and Proof Strategy}\label{sec:mech}
The asymmetry between positive and negative correlation is already visible in an exact
decomposition of the forward before absorption. Write $X=F^{1-\beta}/(1-\beta)$ for the
Lamperti--CEV coordinate and decompose
$W=\rho Z+\orth Z^\perp$, where $\orth:=\sqrt{1-\rho^2}$. The correlated stochastic
integral satisfies $\int_0^t a\,dZ=(a_t-\alpha)/\nu$ exactly, and conditionally on the
volatility path the orthogonal martingale $\int_0^t a\,dZ^\perp$ is a Brownian motion $B$
run on the quadratic-variation clock $V_t=\int_0^ta_s^2\,ds$. Hence, before absorption
(\cref{lem:cev,lem:cond}),
\begin{equation}\label{eq:decomp0}
X_t \;=\; X_0+\underbrace{\tfrac{\rho}{\nu}(a_t-\alpha)}_{\xi_t}\;+\;\orth\,B(V_t)\;-\;b\!\int_0^t\!\frac{a_s^2}{X_s}\,ds,
\qquad b=\frac{\beta}{2(1-\beta)}.
\end{equation}
On the target event $X_T$ is of order $e^n$, where $n:=(1-\beta)k$, while $\xi$ and $B(V)$
are at most of order $\sup_ta_t$ and $\sqrt{V_T}\le\sqrt T\sup_ta_t$. A large terminal value
therefore requires a volatility excursion of logarithmic size $n$, which costs
$ck^2+o(k^2)$. For $\rho\ge0$ the correlation term is bounded below along such an excursion,
survival is cheap, and the unconstrained rate is attained. For $\rho<0$ the correlation
term $\xi_t=-\frac{|\rho|}{\nu}(a_t-\alpha)$ becomes increasingly negative along the very
excursion that creates the wing, and survival of $F$ forces the orthogonal Brownian motion
to stay above the \emph{moving lower barrier}
\[
B(V_t)\;\ge\;C\,(a_t-\alpha)-\frac{X_0}{\orth}\qquad\text{for all } t\le T,
\qquad C:=\frac{|\rho|}{\nu\orth}.
\]

\emph{The cost of survival.} This cost can be computed path by path. Write the volatility
as $a=\alpha e^{\nu\varphi}$, whose cost is $\tfrac12\int_0^T\dot\varphi_t^2\,dt$ to leading
order, and let $B$ track the barrier, $B(V_t)\approx Ca_t$. In operational time
$dv=a^2\,dt$, so $dB/dv=C\dot a/a^2$, and the Cameron--Martin energy of $B$ is
\begin{equation}\label{eq:energy-identity}
\frac12\int\Bigl(\frac{dB}{dv}\Bigr)^2dv
=\frac{C^2}{2}\int_0^T\Bigl(\frac{\dot a_t}{a_t}\Bigr)^2dt
=\nu^2C^2\cdot\frac12\int_0^T\dot\varphi_t^2\,dt.
\end{equation}
For every increasing volatility path, survival therefore costs the fixed fraction
$\nu^2C^2=\rho^2/(1-\rho^2)$ of the volatility energy. The total,
$(1-\rho^2)^{-1}\cdot\frac12\int\dot\varphi^2$, is still minimized by the linear path
$\varphi(t)=nt/(\nu T)$, which gives $c_-k^2=ck^2/(1-\rho^2)$: the effective volatility of
volatility in the right wing is $\nu\orth$. \Cref{lem:A5} is the discrete form of
\eqref{eq:energy-identity}, and in \cref{sec:geometry} the same factor $(1-\rho^2)^{-1}$
appears as $\cosh^2\Rb$.

\emph{The barrier in logarithmic time.} Take the exponential path
$a^*(t)=\alpha e^{\nu\theta t}$, with $\theta$ close to $n/(\nu T)$. Along it
$V_t=(a_t^2-\alpha^2)/(2\nu\theta)$, so the barrier is $\ell\sqrt{v+v_\alpha}$ in operational
time, with $\ell=C\sqrt{2\nu\theta}\approx C\sqrt{2n/T}$ and $v_\alpha=\alpha^2/(2\nu\theta)$
(\cref{fig:mechanism}(a)). In the logarithmic time $s=\ln v$, the process
$U(s)=e^{-s/2}B(e^s)$ is a stationary Ornstein--Uhlenbeck process, and the barrier becomes
essentially the constant level $\ell$ (\cref{fig:mechanism}(b)). Holding $U$ at $\ell$
requires the drift $\ell/2$ that cancels the mean reversion, at Cameron--Martin cost
$\tfrac12(\ell/2)^2=\ell^2/8$ per unit of $s$. The window has length
$\ln(V_T/v_\alpha)\approx2n$, so the total is $\ell^2n/4=C^2n^2/(2T)=\rho^2ck^2/(1-\rho^2)$,
in agreement with \eqref{eq:energy-identity}. The absorbing boundary thus acts on the right
tail through survival conditioning, not through the mass at zero; the model-free
tail-to-wing transfers \cite{Lee,BF,GL} are unaffected.

\begin{figure}[t]
\centering
\includegraphics[width=\textwidth]{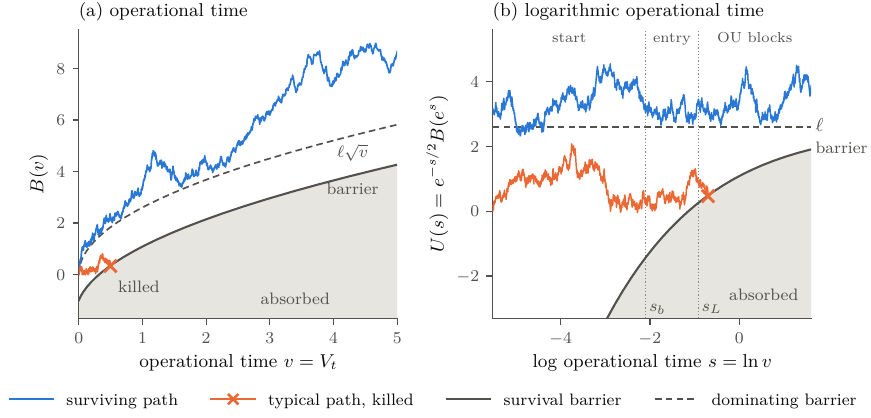}
\caption{Survival as a moving-barrier problem ($\rho<0$; schematic simulation).
(a)~Along the exponential volatility path, survival requires the orthogonal Brownian
motion $B$ to stay above a shifted square-root barrier in operational time $v=V_t$; a
typical path is killed early. (b)~In logarithmic time,
$U(s)=e^{-s/2}B(e^s)$ is a stationary Ornstein--Uhlenbeck process and the barrier rises
to the constant level~$\ell$, which dominates it. The lower-bound construction of
\cref{sec:hover} needs only a fixed lower level before $s_b$, enters the barrier region
between $s_b$ and $s_L$ (\cref{lem:entry}), and then holds $U$ just above $\ell$ by
Ornstein--Uhlenbeck blocks (\cref{lem:block,lem:B4}), at a cost of about $\ell^2/8$ per
unit of~$s$.}
\label{fig:mechanism}
\end{figure}

\emph{Organization of the proof.} \Cref{tab:ledger} shows how the four bounds that make up
\cref{thm:main-tail} are proved and where each cost arises.

\begin{table}[!t]\centering\small
\begin{tabular}{>{\raggedright\arraybackslash}p{2.3cm}>{\raggedright\arraybackslash}p{6.5cm}>{\raggedright\arraybackslash}p{5.2cm}}
\toprule
 & $\rho<0$ (rate $c_-$) & $\rho\ge0$ (rate $c$) \\ \midrule
Upper bound on $P(k)$ &
\Cref{prop:upper}: drop the drift (\cref{lem:A1}); discretize volatility by level-hitting
times (\cref{lem:A3,lem:A4}); bound each conditional Gaussian event by an energy
(\cref{lem:A2}); volatility and barrier energy together are at least $c_-k^2$
(\cref{lem:defer,lem:A5}). &
\Cref{prop:pos-upper}: drop the drift and the survival constraint; a large terminal value
needs a volatility maximum of logarithmic size $n$, at cost $ck^2$ (reflection principle). \\ \addlinespace
Lower bound on $P(k)$ &
\Cref{prop:lower}: volatility tube, cost $(1+\varepsilon_1)^2ck^2$ (\cref{lem:tube});
barrier tracking by Ornstein--Uhlenbeck blocks, cost $\rho^2ck^2/(1-\rho^2)$
(\cref{lem:entry,lem:block,lem:B4}); the bootstrap (\cref{lem:B3}) costs nothing. &
\Cref{prop:pos-lower}: volatility tube, cost $ck^2$; comparison with a killed Bessel
process, cost $O(n)$ (\cref{lem:comparison,lem:bessel}). \\
\bottomrule
\end{tabular}
\caption{The four bounds that make up \cref{thm:main-tail}, the lemmas that prove them, and
their dominant costs. Upper and lower bounds refer to $P(k)$ throughout.}\label{tab:ledger}
\end{table}

\Cref{lem:transfer} turns the rate into \cref{cor:wing}, and \cref{prop:killed-distance},
which is deterministic and independent of \crefrange{sec:upper}{sec:pos}, identifies the
rate with $d_D^2/(2T)$. The upper bound for $\rho<0$ must hold for every volatility path at
once, without presupposing that paths concentrate; its key step, \cref{lem:defer}, shows
that raising $B$ earlier than the barrier requires only costs more. We do not use a
small-noise large-deviation principle: the orthogonal component runs on an operational
time of order $e^{2n}$, and its cost is a long-time Ornstein--Uhlenbeck problem at a large
level rather than a Freidlin--Wentzell scaling.

\subsection{Relation to Prior Work}\label{sec:lit}
The problem lies at the intersection of several established theories, but not inside any
one of their proved regimes. The exact right-wing results for SABR cover the uncorrelated
case \cite{BFL} and the case $\beta=1$ \cite{FZ}; the model-free tail-to-smile transfers
\cite{BF,GL} turn a known tail into a wing but do not determine the tail; and the work on
boundary effects concerns the left wing, survival probabilities and small time. None of
these treats the fixed-maturity extreme right wing of correlated SABR with
$\beta\in(0,1)$; \cref{tab:lit} in \cref{sec:regimes} lists the results we located.
Small-time density expansions \cite{DFJV1,DFJV2} and the geodesic expansions of Paulot
\cite{Paulot} rest on the unrestricted geometry, and Paulot cautions that his expansion is
valid only as long as the geodesic does not reach the boundary; the present result is a
fixed-maturity, extreme-strike manifestation of the same boundary obstruction.

Two statements in the literature bear on the correlated right wing directly. Antonov,
Konikov and Spector \cite[eqs.~(23)--(25)]{AKS13} state for general correlation, as an
intuitive result without proof, a large-strike density decay with the uncorrelated rate
and the wing limit $\nu/(1-\beta)$; \cref{thm:main-tail} shows that for $\rho<0$ the
rate is larger by the factor $(1-\rho^2)^{-1}$. Hagan's formula has the same limit, as
noted after \cref{fig:wing}. The practitioner analyses of the SABR density
\cite{HKLW14,HLW15} address its low-strike side, where the expansion can produce negative
densities; our results concern the high-strike side.

\subsection{Scope and Organization}\label{sec:scope}
Before turning to the proofs, we record precisely what is proved. All limits are taken
for fixed $T>0$, $\beta\in(0,1)$ and $\rho\in(-1,1)$. We do not
prove uniformity in these parameters, a quantitative remainder, results at $\rho=\pm1$ or
$\beta\in\{0,1\}$, or a left-wing theorem. No Dirichlet heat-kernel estimate for the
two-dimensional killed diffusion is used or proved, and \cref{rem:heat-kernel} describes
what such an estimate would have to provide; the only transition density used is the
explicit one of a one-dimensional Bessel process (\cref{lem:bessel}). For $\rho<0$ we state a conjectural $k^{3/2}$ correction
(\cref{conj:subleading}), supported by formal calculations only; no proved
result depends on it. The numerical results of \cref{sec:num} illustrate the
theorems and play no role in the proofs. Selected deterministic steps have also been formalized in
Lean; see the Supplementary Material section.

\Cref{sec:setup} collects the reductions, and \crefrange{sec:upper}{sec:pos} prove the four
bounds of \cref{tab:ledger}. \Cref{sec:geometry} proves the intrinsic-distance identification,
\cref{sec:num} compares the model with Hagan's formula, \cref{sec:disc} discusses the
subleading correction and neighbouring regimes, and \cref{sec:conclusion} concludes;
\cref{app:meas,app:transfer,app:num} contain measurability facts, the proof of the
transfer lemma and the numerical methods.

\section{Preliminaries}\label{sec:setup}

The proofs rest on two reductions. The Lamperti--CEV transform isolates the correlation
term and lets us represent the orthogonal martingale in its quadratic-variation time; a
model-free argument converts logarithmic tail rates into implied-volatility limits.

\emph{The model.} We work on a complete filtered probability space satisfying the usual
conditions and carrying independent standard Brownian motions $(Z,Z^\perp)$. We set
$W=\rho Z+\sqrt{1-\rho^2}\,Z^\perp$ and
\[
a_t=\alpha\exp\!\left(\nu Z_t-\tfrac12\nu^2t\right).
\]
On the open state space $(0,\infty)$, $F$ is the pathwise-unique strong solution of the
first equation in \eqref{eq:sabr} up to the absorption time
\[
\zeta:=\inf\{t\ge0:F_t=0\}\qquad(\inf\varnothing=\infty),
\]
and we take the continuous absorbed extension $F_t=0$ for $t\ge\zeta$. This boundary
condition is part of the model: it selects the absorbing continuation among the possible
post-hitting extensions of the CEV equation.

Throughout, $\beta\in(0,1)$, $\nu>0$, $T>0$ and $\rho\in(-1,1)$ are fixed,
$\orth=\sqrt{1-\rho^2}$, and
\[
C:=\frac{|\rho|}{\nu\orth},\qquad b:=\frac{\beta}{2(1-\beta)},\qquad
n:=(1-\beta)k,\qquad m:=\frac{e^{(1-\beta)k}f_0^{1-\beta}}{1-\beta},
\]
so that $\ln m=n+O(1)$. Constants $c_1,c_2,\dots$ depend only on the model parameters and
on the auxiliary parameters fixed in the current section, never on $k$; their values may
change from section to section. \Cref{tab:params} lists the auxiliary parameters of
\crefrange{sec:upper}{sec:pos}: each is held fixed while $k\to\infty$ and only afterwards
sent to its limit.

\subsection{The Lamperti Decomposition}
The first reduction is a change of coordinates. In the Lamperti coordinate the component
driven by the volatility Brownian motion integrates exactly, leaving an orthogonal
Gaussian martingale and a nonpositive singular drift.

\begin{lemma}[Lamperti Decomposition]\label{lem:cev}
Let $X=F^{1-\beta}/(1-\beta)$. There is an event $\Omega_{\mathrm{CEV}}$ with
$\Prob(\Omega_{\mathrm{CEV}})=1$ such that, for every $\omega\in\Omega_{\mathrm{CEV}}$ and
every $t\in[0,T]$ with $t<\zeta(\omega)$,
\begin{equation}\label{eq:decomp}
X_t \;=\; X_0+\xi_t+\orth\,M_t-b\int_0^t\frac{a_s^2}{X_s}\,ds,
\qquad \xi_t=\frac{\rho}{\nu}(a_t-\alpha),\quad M_t=\int_0^t a_s\,dZ^\perp_s,
\end{equation}
where continuous versions are used throughout. Moreover,
$\{F_T\ge f_0e^k\}=\{X_T\ge m,\ \zeta>T\}$ up to a null set; in particular
$P(k)=\Prob(X_T\ge m,\ \zeta>T)$.
\end{lemma}
\begin{proof}
Apply It\^o's formula to the processes stopped at $\zeta_r:=\inf\{t:X_t\le r\}$ and let
$r\downarrow0$. With $g(F)=F^{1-\beta}/(1-\beta)$ this gives $dX=a\,dW-b(a^2/X)\,dt$
before $\zeta$, using $F^{\beta-1}=((1-\beta)X)^{-1}$. Decompose $W=\rho Z+\orth Z^\perp$
and use $da=\nu a\,dZ$, whence $\int_0^t a\,dZ=(a_t-\alpha)/\nu$. Intersecting the
full-measure events obtained at rational stopping levels and using continuity yields one
event on which the identity holds simultaneously for every admissible $t$. The event
identity follows from the strict monotonicity of $F\mapsto F^{1-\beta}/(1-\beta)$ on
$(0,\infty)$ and the absorption convention, since $F_T=0<f_0e^k$ on $\{\zeta\le T\}$.
\end{proof}

We write
\begin{equation}\label{eq:Xbar}
\bar X_t:=X_0+\xi_t+\orth M_t
\end{equation}
for the process with the singular drift removed.

\subsection{Conditioning on Volatility: Operational Time}
Conditioning on the volatility path completes the first reduction: the orthogonal
martingale then becomes a Brownian motion run on its quadratic-variation clock.

\begin{lemma}[Conditional Brownian Time Change]\label{lem:cond}
For a continuous positive path $y$ let $V_t(y):=\int_0^t y_s^2\,ds$, and let
$t_y:[0,V_T(y)]\to[0,T]$ be the inverse of $t\mapsto V_t(y)$. Conditionally on $\sigma(a)$,
for $\Prob_a$-almost every volatility path $y$, the process $v\mapsto M_{t_y(v)}$ is a
standard Brownian motion on $[0,V_T(y)]$. Precisely, let $\mathsf W$ be Wiener measure on
$C([0,\infty))$, $\Prob_a$ the law of the volatility path $a$ on $C([0,T])$, and
$\Pi_y(w)(t):=w(V_t(y))$; then there is a measurable regular conditional distribution
$\mathsf K$ of the $C([0,T])$-valued random variable $M$ given $\sigma(a)$ such that
$\mathsf K(y,\,\cdot\,)=\mathsf W\circ\Pi_y^{-1}$ for $\Prob_a$-almost every $y$.
\end{lemma}

The jointly measurable versions and the kernel construction are recorded in
\cref{app:meas}. We call $v=V_t$ the \emph{operational time}, write $t(\cdot)$
for the inverse quadratic-variation clock of the volatility path, and set
\[
B(v):=M_{t(v)},\qquad 0\le v\le V_T,
\]
so that $M_t=B(V_t)$. By \cref{lem:cond}, conditionally on $\sigma(a)$, $B$ is a
standard Brownian motion on $[0,V_T]$. ``Conditionally on the volatility path'' always
means evaluation under the kernel $\mathsf K$ for $\Prob_a$-almost every path; the
operational times, thresholds and events entering such evaluations are measurable
functions of the path.

\begin{table}[!b]\centering\small
\begin{tabular}{>{\raggedright\arraybackslash}p{1.75cm}>{\raggedright\arraybackslash}p{1.4cm}>{\raggedright\arraybackslash}p{6.05cm}>{\raggedright\arraybackslash}p{4.3cm}}
\toprule
Parameter & Sections & Role & Limit after $k\to\infty$ \\ \midrule
$\varepsilon$ & \ref{sec:upper}, \ref{sec:pos} & accuracy of the upper bounds & $\varepsilon\downarrow0$ \\
$\eta$, $\kappa$ & \ref{sec:upper} & level spacing of the skeleton; width of the time bins & none: fixed for each $\varepsilon$, with $\eta\le\eta_0(\varepsilon)$, $\kappa\le\kappa_0(\varepsilon)$ \\
$\delta$ & \ref{sec:lower}, \ref{sec:pos} & width of the volatility tube & $\delta\downarrow0$ in \cref{sec:lower}; fixed in \cref{sec:pos} \\
$\varepsilon_1$ & \ref{sec:lower} & overshoot of the volatility, $a^*(T)=\alpha e^{(1+\varepsilon_1)n}$ & $\varepsilon_1\downarrow0$ \\
$\varepsilon_2$ & \ref{sec:lower} & margin above the barrier that absorbs the singular drift & $\varepsilon_2\downarrow0$ \\
$\varepsilon_3$ & \ref{sec:lower} & height of the Ornstein--Uhlenbeck level above the barrier & $\varepsilon_3\downarrow0$ \\
$L$ & \ref{sec:lower} & start $v_L=Lv_\alpha$ of the Ornstein--Uhlenbeck blocks & $L\to\infty$ \\
\bottomrule
\end{tabular}
\caption{The auxiliary parameters of \crefrange{sec:upper}{sec:pos}. In \cref{sec:lower} the
order of the final limits does not matter.}\label{tab:params}
\end{table}

\subsection{Transfer to Implied Volatility}
The second reduction is model-free: once the logarithmic tail rate is known, no further
SABR structure is needed to identify the right-wing limit. \Cref{lem:transfer} is a
special case of the tail-wing formula of Benaim and Friz \cite{BF} and of the general
transfer of Gao and Lee \cite{GL}; we include its short proof to keep the argument
self-contained. Related transfers, in the regime where all moments are finite, are in
\cite{BF08,Gul10,Gul12,CC,Lucic}. The forward $F$ is a nonnegative local martingale, hence
a supermartingale, so $\E F_T\le f_0$; this is all the integrability we use. Write
$\mathcal C(k):=\E(F_T-f_0e^k)^+$ for the undiscounted call price.

\begin{lemma}[Tail-to-Implied-Volatility Transfer]\label{lem:transfer}
If $-\ln P(k)/k^2\to\Lambda\in(0,\infty)$ and $P(k)>0$ for all sufficiently large $k$,
then $\sigma_{\BS}(k)\to 1/\sqrt{2\Lambda T}$.
\end{lemma}
The proof is standard and is given in \cref{app:transfer}.

By \cref{lem:transfer}, \cref{thm:main-tail} implies
\cref{cor:wing}, and hence \cref{cor:resolution}; the positivity
$P(k)>0$ follows from the lower-bound constructions of \cref{sec:lower,sec:pos}. \crefrange{sec:upper}{sec:pos} prove
\cref{thm:main-tail}.

The integrability used above is weaker than what holds, and the stronger fact costs only
a few lines.

\begin{remark}[True Martingale and Moments]\label{rem:martingale}
The forward $F$ is a true martingale, and $\E\sup_{t\le T}F_t^{\,p}<\infty$ for every
$p\ge1$; in particular $\E F_T=f_0$, and $\mathcal C(k)$ is an arbitrage-free price.
Indeed, the drift in \eqref{eq:decomp} is nonpositive, so for $t<\zeta$
\begin{equation*}
X_t\le X_0+\frac{|\rho|}{\nu}\Bigl(\alpha+\sup_{s\le T}a_s\Bigr)
+\orth\sup_{v\le V_T}|B(v)|,
\end{equation*}
while $F_t=0$ for $t\ge\zeta$. Here $\sup_{s\le T}a_s\le\alpha\exp(\nu\sup_{s\le T}Z_s)$
has moments of every order, hence so does $V_T\le T\sup_{s\le T}a_s^2$; and conditionally
on the volatility path, $\sup_{v\le V_T}|B(v)|$ has $p$-th moment $c_pV_T^{p/2}$ by
\cref{lem:cond} and Brownian scaling. Since $F_t=((1-\beta)X_t)^{1/(1-\beta)}$ before
$\zeta$, it follows that $\sup_{t\le T}F_t$ has moments of every order. A continuous
local martingale dominated by an integrable random variable is a martingale.
\end{remark}

\section{Upper Bound for Negative Correlation}\label{sec:upper}

We begin with the upper bound on $P(k)$ for negative correlation. This half of
\cref{thm:main-tail} must rule out every cheaper way of producing the tail,
including paths that raise the orthogonal component before the barrier requires it.

\begin{proposition}[Upper Bound, $\rho<0$]\label{prop:upper}
For $\rho\in(-1,0)$,
\begin{equation}\label{eq:upper}
\liminf_{k\to\infty}\frac{-\ln P(k)}{k^2}\;\ge\;c_-.
\end{equation}
\end{proposition}

The proof has four steps. Removing the nonpositive drift turns survival into a Brownian
moving-barrier event (\cref{lem:A1}). Discretizing the volatility excursion by
level-hitting times turns that event into finitely many linear constraints, whose
conditional probability is bounded by a quadratic energy
(\cref{lem:A2,lem:A3,lem:A4}). A deterministic inequality bounds the combined
volatility and barrier energy below by $c_-k^2$ (\cref{lem:defer,lem:A5}), and
\cref{sec:assembly} sums over the resulting configurations. Throughout this section
$\rho\in(-1,0)$.

\subsection{Drift Removal}
The singular drift is nonpositive before absorption. Dropping it enlarges the joint
survival-and-terminal event and isolates the Gaussian constraints that determine the
upper bound.

\begin{lemma}[Drift Removal]\label{lem:A1}
$\{X_T\ge m,\ \zeta>T\}\subseteq\{\bar X_T\ge m,\ \inf_{t\le T}\bar X_t>0\}$.
\end{lemma}
\begin{proof}
On survival $X>0$, so the drift $-b\,a^2/X$ is nonpositive and
$\bar X_t=X_t+b\int_0^t a^2/X\,ds\ge X_t$ for all $t\le T$.
\end{proof}

\subsection{A Gaussian Convex-Set Bound}
At the level-hitting times introduced below, the relaxed survival and terminal
requirements become linear inequalities for a Gaussian vector. The following projection
estimate converts their conditional probability into a Cameron--Martin energy.

\begin{lemma}[Gaussian Convex-Set Bound]\label{lem:A2}
Let $\mathbf g$ be a centred Gaussian vector in $\mathbb R^p$, $p\ge1$, with symmetric positive
definite covariance matrix $\Sigma$, and let $K\subset\mathbb R^p$ be nonempty, closed and
convex with $0\notin K$. Then
$\Prob(\mathbf g\in K)\le\overline\Phi\bigl(d_\Sigma(0,K)\bigr)\le e^{-\frac12 d_\Sigma(0,K)^2}$,
where $d_\Sigma$ is the Mahalanobis distance and $\overline\Phi:=1-\Phi$.
\end{lemma}
\begin{proof}
Write $\mathbf g=\Sigma^{1/2}G$ with $G$ standard Gaussian and $K'=\Sigma^{-1/2}K$. Let $y^*$ be
the projection of $0$ onto $K'$, so that $|y^*|=d_\Sigma(0,K)$. The supporting hyperplane
at $y^*$ gives $K'\subseteq\{x:\langle x,y^*/|y^*|\rangle\ge|y^*|\}$, and
$\langle G,y^*/|y^*|\rangle\sim N(0,1)$.
\end{proof}

Apply the lemma to $B$ at $\sigma(a)$-measurable operational times
$0<v_1<\dots<v_J\le V_T$, with thresholds $h_1,\dots,h_J$ at these times and a terminal
threshold $H$ at $V_T$. The increments of $B$ are independent, so the squared Mahalanobis
norm of the vector $(B(v_1),\dots,B(v_J),B(V_T))$ at a point $(x_1,\dots,x_{J+1})$ is
$\sum_i(x_i-x_{i-1})^2/(v_i-v_{i-1})$, with $x_0=0$, $v_0=0$ and $v_{J+1}=V_T$. Hence
\begin{equation}\label{eq:QPbound}
\Prob\bigl(B(v_j)\ge h_j\ (1\le j\le J),\ B(V_T)\ge H\ \big|\ \sigma(a)\bigr)\ \le\
\exp\bigl(-\tfrac12\mathrm{QP}\bigr),
\end{equation}
where $\mathrm{QP}$ is the Cameron--Martin energy of the cheapest interpolation through the
constrained nodes:
\begin{equation}\label{eq:QP}
\mathrm{QP}(v;h,H)
:=\min\Bigl\{\sum_{i=1}^{J+1}\mathcal E_{v_i-v_{i-1}}(x_i-x_{i-1}):\
x_0=0,\ x_j\ge h_j\ (1\le j\le J),\ x_{J+1}\ge H\Bigr\},
\end{equation}
and $\mathcal E_s(x):=x^2/s$ for $s>0$. In the applications below $H>0$ for all sufficiently
large $k$, so the constraint set does not contain the origin. Two conventions cover
degenerate cases. A zero increment has energy $\mathcal E_0(0):=0$ and
$\mathcal E_0(x):=+\infty$ for $x\ne0$; thus if $v_J=V_T$ the last two nodes merge, with the
single threshold $\max(h_J,H)$, and \cref{lem:A2} is applied after the merge, so that its
covariance matrix is positive definite. The same formula with no intermediate nodes covers
the terminal-only case $J=0$.

\subsection{Level-Hitting Skeleton}
To retain sharp costs with finitely many constraints, we discretize the increase of
log-volatility by successive level-hitting times. The same skeleton controls the
available quadratic variation, the induced survival thresholds, and the probability of
completing each level crossing quickly.

Fix $\eta\in(0,\tfrac12)$ and set $A_j:=\alpha e^{j\eta}$,
$\tau_j:=\inf\{t: a_t\ge A_j\}$ (so $\tau_0=0$), $J^*:=\max\{j:\tau_j\le T\}$,
$\Delta t_j:=\tau_j-\tau_{j-1}$ and $v_j:=V_{\tau_j}$. Along this skeleton the levels
control the available quadratic variation, and survival becomes a finite list of
threshold conditions at the nodes, with the thresholds
\begin{equation}\label{eq:thresholds}
h_j:=C(A_j-\alpha)-\frac{X_0}{\orth}\quad(1\le j\le J^*),\qquad
H:=\frac{m-X_0}{\orth}+C(a_T-\alpha)\ \ge\ H_{\min}:=\frac{m-X_0}{\orth}-C\alpha.
\end{equation}
The $h_j$ are points of the moving barrier of \cref{sec:mech} (\cref{fig:mechanism}(a)).

\begin{lemma}[Skeleton Constraints]\label{lem:A3}
\textup{(i)} For every $t\in[\tau_{j-1},\tau_j)$, $a_t<A_j$; hence
\begin{equation*}
\Delta v_j:=v_j-v_{j-1}\le A_j^2\Delta t_j=e^{2\eta}A_{j-1}^2\Delta t_j,\qquad
V_T\le \bar V:=2e^{2\eta}A_{J^*}^2 T.
\end{equation*}

\textup{(ii)} On $\{\bar X_T\ge m,\ \inf_{t\le T}\bar X_t>0\}$ one has
$B(v_j)>h_j$ for $1\le j\le J^*$ and $B(V_T)\ge H$, with the thresholds
\eqref{eq:thresholds}.
\end{lemma}
\begin{proof}
(i) is immediate from the definition of first passage; the cap on $V_T$ holds because
\begin{equation*}
a<A_{J^*+1}=e^{\eta}A_{J^*}\qquad\text{on }[0,T].
\end{equation*}
(ii) Since $a_{\tau_j}=A_j$ by continuity and
$\xi_t=-C\orth(a_t-\alpha)$, the inequality $\bar X_{\tau_j}>0$ reads
$\orth B(v_j)>C\orth(A_j-\alpha)-X_0$, and $\bar X_T\ge m$ reads
$\orth B(V_T)\ge m-X_0-\xi_T$.
\end{proof}

Conversely, crossing a level quickly is costly for the volatility, as the reflection
principle shows.

\begin{lemma}[Level-Hitting-Time Bound]\label{lem:A4}
On $\{\tau_{j-1}<\infty\}$, for every $s>0$,
\begin{equation*}
\Prob\bigl(\tau_j<\infty,\ \tau_j\le\tau_{j-1}+s\,\big|\,\mathcal F_{\tau_{j-1}}\bigr)\ \le\ 2\exp\Bigl(-\frac{\eta^2}{2\nu^2 s}\Bigr).
\end{equation*}
Consequently, for any thresholds $\bar t_1,\dots,\bar t_J>0$, by iterated conditioning,
\begin{equation*}
\Prob(\Delta t_j\le\bar t_j\ \text{for all } j\le J)\ \le\ 2^{J}\exp\Bigl(-\sum_{j\le J}\frac{\eta^2}{2\nu^2\bar t_j}\Bigr).
\end{equation*}
\end{lemma}
\begin{proof}
By the strong Markov property, from $\tau_{j-1}$ the log-volatility runs as
$\nu B'_u-\nu^2u/2$ for a fresh Brownian motion $B'$; rising by $\eta$ within time $s$
forces $\sup_{u\le s}\nu B'_u\ge\eta$, and the reflection principle gives the bound.
\end{proof}

\subsection{Deferred Displacement and the Energy Inequality}
The preceding reductions leave a deterministic question: can a path bring its combined
volatility and Gaussian cost below $c_-k^2$, by redistributing the volatility increase or
by raising the Gaussian coordinate before the barrier requires it? The step costs show why
the answer should be no. In level step $j$ the volatility pays $\eta^2/(2\nu^2\Delta t_j)$
(\cref{lem:A4}), and $B$ pays about $(\Delta x_j)^2/(2A_{j-1}^2\Delta t_j)$. Once the
barrier binds, $\Delta x_j\ge h_j-h_{j-1}\approx C\eta A_{j-1}$, so the step costs at least
$\eta^2(1+\nu^2C^2)/(2\nu^2\Delta t_j)$, and Cauchy--Schwarz with
$\sum_j\Delta t_j\le T$ and $J\eta\approx n$ gives $c_-k^2$; this is the discrete form of
\eqref{eq:energy-identity}. Two escapes must be closed. A path may raise $B$ before the
barrier requires it, and \cref{lem:defer}, an elementary summation-by-parts bound for the
part of the Gaussian cost that is linear in the displacement, shows that this only costs
more. A path may also stop short of level $n$, and then the available quadratic variation
is far too small for $B$ to reach the terminal threshold.

\begin{lemma}[Deferred Displacement]\label{lem:defer}
Let $w_1\ge w_2\ge\dots\ge w_J\ge0$ and $h_1\le h_2\le\dots\le h_J$, write
$h_j^+:=\max(h_j,0)$, and put $h_0^+:=0$. For all $x_0=0,x_1,\dots,x_J\in\mathbb R$ with
$x_j\ge h_j$ for $1\le j\le J$,
\[
\sum_{j=1}^J w_j\,|x_j-x_{j-1}|\;\ge\;\sum_{j=1}^J w_j\,\bigl(h_j^+-h_{j-1}^+\bigr),
\]
with equality for the maximally deferred choice $x_j=h_j^+$.
\end{lemma}
\begin{proof}
Let $U_0:=0$ and $U_j:=\sum_{i\le j}(x_i-x_{i-1})^+$. Then $U_j\ge x_j\ge h_j$ and
$U_j\ge0$, so $U_j\ge h_j^+$. Summation by parts with $w_{J+1}:=0$ gives
\[
\begin{aligned}
\sum_{j=1}^J w_j|x_j-x_{j-1}|
&\ \ge\ \sum_{j=1}^J w_j(U_j-U_{j-1})
=\sum_{j=1}^J (w_j-w_{j+1})\,U_j\\
&\ \ge\ \sum_{j=1}^J (w_j-w_{j+1})\,h_j^+
=\sum_{j=1}^J w_j\bigl(h_j^+-h_{j-1}^+\bigr),
\end{aligned}
\]
since $w_j-w_{j+1}\ge0$. For $x_j=h_j^+$ the increments are nonnegative and both
inequalities are equalities.
\end{proof}

With this bound for the linear part, the combined volatility and Gaussian energy can be
bounded below uniformly over all allocations of time and quadratic variation.

Call increments \emph{admissible}, for given $J\ge1$, $\eta$ and $\kappa$, if
\begin{equation}\label{eq:admissible}
\textup{(a)}\ \sum_{j=1}^J\Delta t_j\le(1+\kappa)T,\quad
\textup{(b)}\ 0<\Delta v_j\le e^{2\eta}A_{j-1}^2\Delta t_j,\quad
\textup{(c)}\ 0\le\Delta v_\infty\le\bar V:=2e^{2\eta}A_J^2T,
\end{equation}
where $\Delta v_\infty$ is a terminal increment, and write $\mathrm{QP}$ for \eqref{eq:QP}
with the operational-time increments $(\Delta v_1,\ldots,\Delta v_J,\Delta v_\infty)$, the
thresholds $h_j$ of \eqref{eq:thresholds}, and a terminal threshold $H\ge H_{\min}$.

\begin{lemma}[Energy Inequality]\label{lem:A5}
For every $\varepsilon\in(0,1)$ there are $\eta_0\in(0,\tfrac12)$ and $\kappa_0\in(0,1)$,
depending only on $\varepsilon$ and the model parameters, such that for all
$\eta\le\eta_0$ and $\kappa\le\kappa_0$ there is $k_0$ with the following property: for all
$k\ge k_0$, all $J\ge1$, all increments admissible in the sense of \eqref{eq:admissible} and
all $H\ge H_{\min}$,
\[
\sum_{j\le J}\frac{\eta^2}{2\nu^2\Delta t_j}\;+\;\tfrac12\,\mathrm{QP}
\;\ge\;(1-\varepsilon)\,c_-k^2.
\]
\end{lemma}

\begin{proof}
Put $\varepsilon':=\varepsilon/4$, so that $(1-\varepsilon')^3\ge1-\varepsilon$; we prove the
bound with the factor $(1-\varepsilon')^3$. $\mathrm{QP}$ is nonincreasing in each
operational-time increment, so we may replace
$\Delta v_j$ by its upper bound and the terminal increment by $\bar V$; this remains valid
when the actual terminal increment is zero, because the convention in \eqref{eq:QP} is
the lower-semicontinuous limit as that increment decreases to zero. Write the QP variables
through their increments $\Delta x_j:=x_j-x_{j-1}$ ($j\le J$) and
$\Delta x_\infty:=x_{J+1}-x_J$, so that $x_j\ge h_j$ and $x_J+\Delta x_\infty\ge H$. Per
interval,
\[
\frac{\eta^2}{2\nu^2\Delta t_j}+\frac{(\Delta x_j)^2}{2e^{2\eta}A_{j-1}^2\Delta t_j}
=\frac{z_j}{\Delta t_j},\qquad
z_j:=\frac{\eta^2}{2\nu^2}+\frac{e^{-2\eta}(\Delta x_j)^2}{2A_{j-1}^2},
\]
and by Cauchy--Schwarz
$\sum_j z_j/\Delta t_j\ge(\sum_j\sqrt{z_j})^2/\sum_j\Delta t_j\ge(\sum_j\sqrt{z_j})^2/((1+\kappa)T)$;
the terminal term $(\Delta x_\infty)^2/(2\bar V)$ is kept aside as a separate nonnegative
summand. Put $r:=J\eta/n$. All $O$-terms below are uniform in $J$ and in the admissible
increments; the coefficient of $O(\eta+\kappa)$ is uniform for small $\eta,\kappa$, and
once $\eta,\kappa$ are fixed the coefficient of $O(1/n)$ does not depend on $k$.

\emph{Main case $r\ge1-\varepsilon'$.} Drop the terminal term. For any $\vartheta_1,\vartheta_2\ge0$ with
$\vartheta_1^2+\vartheta_2^2=1$,
$\sqrt{z_j}\ge \vartheta_1\eta/(\sqrt2\nu)+\vartheta_2e^{-\eta}|\Delta x_j|/(\sqrt2 A_{j-1})$. The weights
$1/A_{j-1}$ decrease and the thresholds $h_j$ increase, so \cref{lem:defer} gives
\[
\sum_{j\le J}\frac{|\Delta x_j|}{A_{j-1}}\ \ge\ \sum_{j\le J}\frac{h_j^+-h_{j-1}^+}{A_{j-1}}\ \ge\ C(e^\eta-1)\,J_{\mathrm{bind}},
\qquad J_{\mathrm{bind}}:=\#\{j\le J:h_{j-1}\ge0\},
\]
where $h_0:=-X_0/\orth$ continues the formula of \cref{lem:A3}; indeed
$(h_j-h_{j-1})/A_{j-1}=C(e^\eta-1)$ whenever $h_{j-1}\ge0$. If $A_b:=\alpha+X_0/(C\orth)$
denotes the level at which $h$ becomes nonnegative, then
$J_{\mathrm{bind}}\ge J-\eta^{-1}\ln(A_b/\alpha)-1$. Raising the Gaussian coordinate
earlier than the barrier requires can therefore only increase this sum. Hence
\[
\sum_j\sqrt{z_j}\ \ge\ (J\eta)\Bigl[\frac{\vartheta_1}{\sqrt2\,\nu}+\frac{\vartheta_2C}{\sqrt2}\bigl(1-O(\eta)-O(1/n)\bigr)\Bigr],
\]
and optimizing the direction,
$\sup_{\vartheta_1^2+\vartheta_2^2=1}[\,\cdot\,]=\sqrt{1/(2\nu^2)+C^2/2}\,(1-O(\eta)-O(1/n))$. With
$J\eta=rn\ge(1-\varepsilon')n$ and $1+\nu^2C^2=(1-\rho^2)^{-1}$,
\[
\begin{aligned}
\frac{(\sum_j\sqrt{z_j})^2}{(1+\kappa)T}
&\ \ge\ (1-\varepsilon')^2\,\frac{n^2}{2\nu^2T(1-\rho^2)}\bigl(1-O(\eta+\kappa)-O(1/n)\bigr)\\
&\ =\ (1-\varepsilon')^2c_-k^2\bigl(1-O(\eta+\kappa)-O(1/n)\bigr).
\end{aligned}
\]

\emph{Case $r<1-\varepsilon'$.} Here $H\ge m/(2\orth)$ for large $k$. Since
$x_J+\Delta x_\infty\ge H$, either $\Delta x_\infty\ge H/2$ or
$x_J=\sum_{j\le J}\Delta x_j\ge H/2$. In the former case
\[
\frac{(\Delta x_\infty)^2}{2\bar V}\ \ge\ \frac{m^2}{64\,\orth^2e^{2\eta}A_J^2T}\ \ge\ \frac{e^{2(1-r)n}}{c_3T}\ \ge\ \frac{e^{2\varepsilon' n}}{c_3T};
\]
in the latter, taking $(\vartheta_1,\vartheta_2)=(0,1)$,
$\sum_j\sqrt{z_j}\ge (e^{-\eta}/\sqrt2)(H/2)/A_J\ge e^{(1-r)n}/c_4$, and squaring gives at
least $e^{2\varepsilon'n}/(c_5T)$. Both dominate $c_-k^2$ for $k\ge k_0$.

Finally choose $\eta_0$ and $\kappa_0$ so that the $O(\eta+\kappa)$ loss in the main
case is at most $\varepsilon'/2$. For each fixed $\eta\le\eta_0$ and
$\kappa\le\kappa_0$, enlarge $k_0$ so that the $O(1/n)$ loss is at most
$\varepsilon'/2$, $H\ge m/(2\orth)$, and the exponential bounds of the second case
dominate $c_-k^2$. This gives the factor $1-\varepsilon'$ uniformly over $J$ and all
admissible increments.
\end{proof}

Exact tracking of the barrier along the exponential volatility path attains this bound to
leading order; it is the path used for the lower bound in \cref{sec:lower}. The proof never
needs the minimizer of the discrete problem: the directional bound and \cref{lem:defer}
apply to all admissible allocations at once.

\subsection{Assembly}\label{sec:assembly}
It remains to sum the cellwise bounds without losing the quadratic exponent. We first
remove exceptionally high volatility levels and very short passage times, then bin the
remaining passage times so that \cref{lem:A5} applies uniformly at a combinatorial
cost of only $e^{o(n^2)}$.

\begin{proof}[Proof of \cref{prop:upper}]
By \cref{lem:cev,lem:A1},
$P(k)\le\Prob(\bar X_T\ge m,\ \inf_{t\le T}\bar X_t>0)$. Fix
$\varepsilon\in(0,1)$ and $\eta,\kappa$ in the ranges of \cref{lem:A5}; let
$r_{\mathrm{cap}}:=\lceil(1-\rho^2)^{-1/2}\rceil+1$,
$J_{\max}:=\lceil r_{\mathrm{cap}}n/\eta\rceil$,
$N_\kappa:=\lceil\ln(n^{10})/\ln(1+\kappa)\rceil$ (any exponent larger than $2$ in place of
$10$ would do), and define time bins
$\mathcal B_i:=(T(1+\kappa)^{-i-1},\,T(1+\kappa)^{-i}]$ for $i<N_\kappa$ and
$\mathcal B_{N_\kappa}:=(0,\,T(1+\kappa)^{-N_\kappa}]$.

\emph{(a) High volatility.}
$\{J^*\ge J_{\max}\}\subseteq\{\sup_{t\le T}(\nu Z_t-\nu^2t/2)\ge r_{\mathrm{cap}}n-\eta\}$,
so by the reflection principle its probability is at most
$\exp(-(r_{\mathrm{cap}}n-\eta)^2/(2\nu^2T))\le e^{-c_-k^2}$ for large $k$, since
$r_{\mathrm{cap}}^2\ge(1-\rho^2)^{-1}$.

\emph{(b) No level crossing, $J^*=0$.} Then $a<\alpha e^{\eta}$ on $[0,T]$,
$V_T\le\alpha^2e^{2\eta}T$, and the one-node case of \cref{lem:A2} bounds the
conditional probability by $\exp(-c_6 m^2)$.

\emph{(c) Configurations.} For $1\le J^*<J_{\max}$, partition the event according to the
configuration $\mathbf i=(J^*,(i_j)_{j\le J^*})$, where $i_j$ is the bin containing
$\Delta t_j$, and write $\bar t_j:=\sup\mathcal B_{i_j}$ for the upper end of that bin. By
\cref{lem:A4} alone, each cell with some $i_j=N_\kappa$ has probability at most
\begin{equation*}
2\exp\Bigl(-\frac{\eta^2(1+\kappa)^{N_\kappa}}{2\nu^2T}\Bigr)\le e^{-n^9}\qquad\text{for large }k,
\end{equation*}
and the number of configurations is at most
\begin{equation*}
\sum_{J<J_{\max}}(N_\kappa+1)^{J}\le\exp\Bigl(O\Bigl(\frac n\eta\Bigl(\ln\ln n+\ln\frac1\kappa\Bigr)\Bigr)\Bigr)=e^{o(n^2)},
\end{equation*}
so these cells contribute at most $e^{-n^9+o(n^2)}$. Otherwise
$\bar t_j\le(1+\kappa)\Delta t_j$, hence $\sum_j\bar t_j\le(1+\kappa)T$, and on each
cell the probability of the volatility configuration is at most
$2^{J^*}\exp(-\sum_j\eta^2/(2\nu^2\bar t_j))$ by \cref{lem:A4} with thresholds $\bar t_j$,
while the conditional Gaussian probability is at most $\exp(-\frac12\mathrm{QP})$ by
\eqref{eq:QPbound} and \cref{lem:A3}, with the
operational-time increments bounded through $\bar t_j$ (larger increments decrease
$\mathrm{QP}$). By \cref{lem:A5} with $\Delta t_j:=\bar t_j$, each such cell therefore
contributes at most
\begin{equation*}
\Prob\bigl(\text{cell }\mathbf i,\ \bar X_T\ge m,\ \inf_{t\le T}\bar X_t>0\bigr)
\ \le\ 2^{J^*}\exp\Bigl(-\sum_{j\le J^*}\frac{\eta^2}{2\nu^2\bar t_j}-\tfrac12\mathrm{QP}\Bigr)
\ \le\ 2^{J^*}e^{-(1-\varepsilon)c_-k^2}.
\end{equation*}

\emph{(d) Sum.} Combining,
\[
P(k)\ \le\ e^{-c_-k^2}+e^{-c_6m^2}+e^{-n^9+o(n^2)}+e^{o(n^2)}\, e^{-(1-\varepsilon)c_-k^2}.
\]
Letting $k\to\infty$ and then $\varepsilon\downarrow0$, with admissible $\eta,\kappa$ fixed
for each $\varepsilon$, proves the proposition.
\end{proof}

\section{Lower Bound for Negative Correlation}\label{sec:lower}

We now construct surviving paths that produce the tail at the rate $c_-$, which gives the
reverse inequality.

\begin{proposition}[Lower Bound, $\rho<0$]\label{prop:lower}
For $\rho\in(-1,0)$,
\[
\limsup_{k\to\infty}\frac{-\ln P(k)}{k^2}\;\le\;c_-.
\]
\end{proposition}

The construction separates the two leading costs, in four steps. A Girsanov change of
measure confines volatility to a tube around an exponential path (\cref{lem:tube}).
Conditionally on that tube, an event $E_B$ for the orthogonal Brownian motion, which
enters and stays above the induced square-root barrier (\cref{fig:mechanism}), implies
survival and the terminal event once the singular drift is controlled (\cref{lem:B3}).
Its conditional probability comes from an explicit entry segment followed by
Ornstein--Uhlenbeck survival blocks (\cref{lem:entry,lem:block,lem:B4}), and
\cref{sec:lower-assembly} multiplies the costs. Throughout this section $\rho\in(-1,0)$.

Fix $\varepsilon_1,\varepsilon_2,\delta\in(0,\tfrac14)$, set
$\theta:=n(1+\varepsilon_1)/(\nu T)$ and $a^*(t):=\alpha e^{\nu\theta t}$, and let
$V^*(t):=\int_0^t a^*(s)^2ds$. Here $\delta$ is the width of the volatility tube,
$\varepsilon_2$ is a margin above the barrier that absorbs the singular drift, and
$\varepsilon_1$ is an overshoot of the volatility: $a^*(T)=\alpha e^{(1+\varepsilon_1)n}$ is
much larger than $m$, so that tracking the barrier up to time $T$ already implies the
terminal constraint. \Cref{tab:params} collects all auxiliary parameters. The exact identity
$a^*(t)=\sqrt{2\nu\theta V^*(t)+\alpha^2}$ shows that \emph{the exponential path is a
shifted square-root profile in its own quadratic-variation time}.

\subsection{Volatility Tube}
We first confine volatility near the exponential path suggested by the leading
variational balance. The logarithmic tube below has the unconstrained quadratic cost and
distorts the volatility level and its quadratic-variation clock only by fixed factors.

\begin{lemma}[Volatility Tube]\label{lem:tube}
For $\theta>0$ and $\delta\in(0,1)$ let $a^*(t)=\alpha e^{\nu\theta t}$ and
\begin{equation*}
A_{\mathrm{vol}}:=\Bigl\{\sup_{t\le T}|\ln a_t-\ln a^*(t)|\le\delta\Bigr\}.
\end{equation*}
Then
\[
\Prob(A_{\mathrm{vol}})\ \ge\ p(\delta)\,\exp\Bigl(-\frac{(\theta+\nu/2)^2T}{2}-\frac{(\theta+\nu/2)\,\delta}{\nu}\Bigr),
\qquad p(\delta):=\Prob\Bigl(\sup_{t\le T}|\nu Z_t|\le\delta\Bigr)>0.
\]
On $A_{\mathrm{vol}}$, $e^{-\delta}a^*(t)\le a_t\le e^{\delta}a^*(t)$ and
$e^{-2\delta}V^*(t)\le V_t\le e^{2\delta}V^*(t)$ for all $t\le T$.
\end{lemma}
\begin{proof}
Let $\bar\theta:=\theta+\nu/2$, and let $\widetilde\Prob$ be the measure, given by Girsanov's
theorem, under which $\widetilde Z_t:=Z_t-\bar\theta t$ is a Brownian motion. The extra $\nu/2$
absorbs the It\^o term, so that exactly
\begin{equation*}
\ln a_t=\ln\alpha+\nu Z_t-\tfrac12\nu^2t=\ln a^*(t)+\nu\widetilde Z_t.
\end{equation*}
Hence $\widetilde\Prob(A_{\mathrm{vol}})=p(\delta)$, and on $A_{\mathrm{vol}}$, where
$|\nu\widetilde Z_T|\le\delta$,
\begin{equation*}
\frac{d\Prob}{d\widetilde\Prob}=\exp\Bigl(-\bar\theta\widetilde Z_T-\frac{\bar\theta^2T}{2}\Bigr)\ \ge\
\exp\Bigl(-\frac{\bar\theta\delta}{\nu}-\frac{\bar\theta^2T}{2}\Bigr).
\end{equation*}
The bounds on $a_t$ and $V_t$ are immediate.
\end{proof}

With $\theta=n(1+\varepsilon_1)/(\nu T)$ the lemma gives
\begin{equation}\label{eq:tube}
\Prob(A_{\mathrm{vol}})\ \ge\ \exp\Bigl(-\frac{n^2(1+\varepsilon_1)^2}{2\nu^2T}-O(n)\Bigr).
\end{equation}

\subsection{Target Event for the Orthogonal Component}
Conditionally on a volatility path in the tube, the orthogonal Brownian motion must keep
$\bar X$ positive before the barrier becomes active, stay above the moving barrier
afterwards, and leave enough terminal margin to reach $m$ once the singular drift is
restored.

Let $\hat a(v):=a_{t(v)}$ denote volatility in operational time. Set
$c_*:=\nu X_0/(2(1+\varepsilon_2)|\rho|)$; below the level $\alpha+c_*$ the correlation term
is at most $X_0/(2(1+\varepsilon_2))$ in absolute value, so no barrier is needed there. Let
$v_b$ be the first operational time at which $\hat a$ exceeds $\alpha+c_*$; since $\hat a(0)=\alpha$ and $\hat a$ is continuous,
$v_b>0$. Define
\[
\begin{aligned}
E_B:= {}&\Bigl\{B(v)\ge(1+\varepsilon_2)C\hat a(v)\ \forall v\in[v_b,V_T]\Bigr\}
\cap\Bigl\{B\ge-\tfrac{X_0}{2\orth}\text{ on }[0,v_b]\Bigr\}\\
&{}\cap\Bigl\{B(V_T)\ge\tfrac{1+\varepsilon_2}{\orth}\bigl(m+\tfrac{|\rho|}{\nu}a_T\bigr)\Bigr\}.
\end{aligned}
\]
On the tube, with $v_\alpha:=\alpha^2/(2\nu\theta)$, the identity
$a^*(t)=\sqrt{2\nu\theta V^*(t)+\alpha^2}$ and \cref{lem:tube} give the two-sided
bound
\begin{equation}\label{eq:ahat-sqrt}
e^{-2\delta}\sqrt{2\nu\theta\,(v+v_\alpha)}\ \le\ \hat a(v)\ \le\ e^{2\delta}\sqrt{2\nu\theta\,(v+v_\alpha)},
\end{equation}
so the moving barrier $(1+\varepsilon_2)C\hat a$ is dominated by
$\ell_+\sqrt{v+v_\alpha}$, where $\ell_+:=(1+\varepsilon_2)e^{2\delta}C\sqrt{2\nu\theta}$.
Evaluating \eqref{eq:ahat-sqrt} at $v_b$ shows $c_{\mathrm b,1}/\theta\le v_b\le c_{\mathrm b,2}/\theta$
with $c_{\mathrm b,1}>0$, provided $\delta<\delta_0:=\tfrac12\ln(1+c_*/\alpha)$, which we
assume henceforth. The constants of this section may depend on the fixed $\delta$, which
is harmless because $\delta\downarrow0$ only after $k\to\infty$.

\subsection{Bootstrap Control of the Singular Drift}
The slack factors in $E_B$ were chosen to dominate the accumulated singular drift. The
next lemma turns the tube-and-barrier event into a surviving path that reaches the
terminal target.

\begin{lemma}[Bootstrap]\label{lem:B3}
There is $k_0=k_0(\varepsilon_1,\varepsilon_2,\delta)$ such that for $k\ge k_0$,
\begin{equation*}
A_{\mathrm{vol}}\cap E_B\ \subseteq\ \{\zeta>T,\ X_T\ge m\}.
\end{equation*}
\end{lemma}
\begin{proof}
Recall $\bar X=X_0+\xi+\orth B(V)$ from \eqref{eq:Xbar}, and work on
$A_{\mathrm{vol}}\cap E_B$.

\emph{Positivity of $\bar X$.} For $v\ge v_b$,
$\orth B\ge(1+\varepsilon_2)\frac{|\rho|}{\nu}\hat a$, so
\[
\bar X\ge X_0-\frac{|\rho|}{\nu}(\hat a-\alpha)+(1+\varepsilon_2)\frac{|\rho|}{\nu}\hat a
=X_0+\frac{|\rho|}{\nu}(\varepsilon_2\hat a+\alpha)\ge\varepsilon_2\frac{|\rho|}{\nu}\hat a.
\]
For $v\le v_b$, $\hat a\le\alpha+c_*$ gives
$\xi=-\tfrac{|\rho|}{\nu}(\hat a-\alpha)\ge-\tfrac{|\rho|}{\nu}c_*=-\tfrac{X_0}{2(1+\varepsilon_2)}$,
and with $B\ge-X_0/(2\orth)$,
\[
\bar X\ge X_0-\frac{X_0}{2(1+\varepsilon_2)}-\frac{X_0}{2}=\frac{\varepsilon_2X_0}{2(1+\varepsilon_2)}>0.
\]

\emph{Drift control.} Let $\tau_{1/2}:=\inf\{t: X_t\le\bar X_t/2\}$. On
$[0,\tau_{1/2})$, by the two cases above,
\[
X\ \ge\ \frac{\bar X}2\ \ge\ \frac{\varepsilon_2X_0}{4(1+\varepsilon_2)}\quad(v\le v_b),
\qquad
X\ \ge\ \frac{\bar X}2\ \ge\ \frac{\varepsilon_2}{2}\frac{|\rho|}{\nu}\hat a\quad(v\ge v_b);
\]
for $v\le v_b$ one has $a\le\alpha+c_*$, so $a^2\le(\alpha+c_*)a$, and for $v\ge v_b$,
$a^2/\hat a=a$. Hence the drift integral satisfies
\[
\mathcal I_t:=b\int_0^t\frac{a_s^2}{X_s}\,ds\le c_2(\varepsilon_2)\int_0^t a_s\,ds\le c_2(\varepsilon_2)\,\frac{e^{\delta}a^*(t)}{\nu\theta},
\]
so $\mathcal I_t/\bar X_t\le c_3(\varepsilon_2,\delta)/\theta\to0$ uniformly in $t$. In particular
$\mathcal I<\bar X/2$ throughout for large $k$, which is incompatible with $\tau_{1/2}\le T$. Hence
$X\ge\bar X/2>0$ on $[0,T]$, that is, $\zeta>T$.

\emph{Terminal value.} The terminal constraint in $E_B$ and the drift bound give, for
$k$ large,
\begin{equation*}
\begin{aligned}
\bar X_T&\ge(1+\varepsilon_2)m+\varepsilon_2\tfrac{|\rho|}{\nu}a_T+X_0+\tfrac{|\rho|}{\nu}\alpha,\\
\mathcal I_T&\le c_2(\varepsilon_2)e^{2\delta}a_T/(\nu\theta)=o(a_T),\\
X_T&=\bar X_T-\mathcal I_T\ge(1+\varepsilon_2)m+\varepsilon_2\tfrac{|\rho|}{\nu}a_T-o(a_T)\ge m.
\end{aligned}
\end{equation*}
In the first line the $\xi_T$ term cancels all but the $\varepsilon_2$-fraction of the
barrier margin; in the second, $\mathcal I_T=o(a_T)$ because $\theta\to\infty$.
\end{proof}

\subsection{Barrier Tracking}\label{sec:hover}
It remains to show that $E_B$ has the required conditional probability. Under the
logarithmic change of operational time, a square-root Brownian barrier becomes a constant
barrier for a stationary Ornstein--Uhlenbeck process (\cref{fig:mechanism}(b)).
This representation does not control the initial entry, so the construction first
follows an explicit Cameron--Martin path segment and then concatenates
Ornstein--Uhlenbeck survival blocks.

Fix $\varepsilon_3\in(0,\tfrac14)$ and $L\ge L_0$, where $L_0=L_0(\text{model parameters})$
is chosen so that $v_b\le Lv_\alpha$ on the tube (possible since $v_b\le c_{\mathrm b,2}/\theta$
and $v_\alpha=\alpha^2/(2\nu\theta)$). Set
\[
\begin{aligned}
v_L&:=Lv_\alpha, & s_b&:=\ln v_b, & s_L&:=\ln v_L, & s_T&:=\ln V_T,\\
\ell_L&:=\ell_+\sqrt{1+1/L}, &
u_0&:=\ell_L(1+\varepsilon_3), &
h_L&:=\ell_+\sqrt{(L+1)v_\alpha}.
\end{aligned}
\]
Here $h_L=(1+\varepsilon_2)e^{2\delta}C\alpha\sqrt{L+1}$ and $u_0\sqrt{v_L}=(1+\varepsilon_3)h_L$
are constants independent of $k$, while $u_0\asymp\ell_+\asymp\sqrt n$. We use the
constant barrier for $U$ only on $[v_L,V_T]$, where $v+v_\alpha\le(1+1/L)v$, and bridge
$[0,v_L]$ explicitly; \cref{fig:mechanism}(b) marks the times $s_b$ and $s_L$. Two elementary lemmas isolate the probabilistic inputs.

\begin{lemma}[Brownian Barrier Entry]\label{lem:entry}
Conditionally on a volatility path in the tube, for all $k\ge k_0(L,\varepsilon_3,\varepsilon_2,\delta,\varepsilon_1)$,
\[
\Prob\Bigl(B\ge-\tfrac{X_0}{2\orth}\ \text{on}\ [0,v_b];\ \ B\ge h_L\ \text{on}\ [v_b,v_L];\ \ B(v_L)\in\bigl[u_0\sqrt{v_L},(u_0+1)\sqrt{v_L}\bigr]\Bigr)\ \ge\ e^{-c_L\ell_+^2},
\]
with $c_L=c(L,\varepsilon_3,\text{params})$.
\end{lemma}
\begin{proof}
Let $\phi$ be piecewise linear: $\phi\equiv0$ on $[0,v_b/2]$, rising linearly to
$(u_0+\tfrac12)\sqrt{v_L}$ at $v_b$, and constant on $[v_b,v_L]$. On
$\{\sup_{v\le v_L}|B-\phi|\le\tfrac12\sqrt{v_L}\}$ all three constraints hold: on
$[0,v_b]$, $B\ge-\tfrac12\sqrt{v_L}\ge-X_0/(2\orth)$ for $k$ large (as $v_L\to0$); on
$[v_b,v_L]$, $B\ge u_0\sqrt{v_L}=(1+\varepsilon_3)h_L\ge h_L$; and
$B(v_L)\in[u_0\sqrt{v_L},(u_0+1)\sqrt{v_L}]$. The localized Cameron--Martin shift along
$\phi$ and Brownian scaling,
$\Prob\bigl(\sup_{v\le v_L}|\widetilde B|\le\tfrac12\sqrt{v_L}\bigr)=\Prob\bigl(\sup_{u\le1}|B_u|\le\tfrac12\bigr)=:p_*>0$,
give
\[
\Prob(\cdot)\ \ge\ p_*\exp\Bigl(-\frac{(u_0+\tfrac12)^2v_L}{v_b}-\frac{2(u_0+\tfrac12)\,v_L}{v_b}\Bigr)\ \ge\ e^{-c_L\ell_+^2},
\]
since
\begin{equation*}
\frac12\int\dot\phi^2=\frac{(u_0+\frac12)^2v_L}{v_b},\qquad
\Bigl|\int\dot\phi\,d\widetilde B\Bigr|\le|\dot\phi|_\infty\sqrt{v_L}=\frac{2(u_0+\frac12)\,v_L}{v_b}
\end{equation*}
on the localization event, while $v_L/v_b\le Lv_\alpha\theta/c_{\mathrm b,1}$ is a constant
depending only on $L$ and the parameters, and $u_0\asymp\ell_+$.
\end{proof}

After the entry, the path is carried through blocks of logarithmic length between
$\tfrac12$ and $1$; the next lemma gives the cost of a single block.

\begin{lemma}[Ornstein--Uhlenbeck Survival Block]\label{lem:block}
Let $U$ solve $dU=-\tfrac12U\,ds+dW^{\mathrm{OU}}$. For each $\varepsilon_3\in(0,\tfrac14)$ there exist
$p_1>0$ and $\ell_0$ such that for all $\ell\ge\ell_0$, all durations
$\Delta s\in[\tfrac12,1]$, and all
$x\in I_\ell:=[\ell(1+\varepsilon_3),\ell(1+\varepsilon_3)+1]$,
\[
\Prob_x\bigl(U_s\ge\ell\ (0\le s\le\Delta s),\ U_{\Delta s}\in I_\ell\bigr)\ \ge\ p_1\exp\Bigl(-\frac{\ell^2(1+\varepsilon_3)^2}{8}\,\Delta s-\ell^{3/2}\Bigr).
\]
\end{lemma}
\begin{proof}
Write $u_\ell:=\ell(1+\varepsilon_3)$ and tilt the driving Brownian motion by the
constant drift $u_\ell/2$: under the tilt, $dU=-\tfrac12(U-u_\ell)\,ds+d\widetilde W^{\mathrm{OU}}$
is an Ornstein--Uhlenbeck process centred at $u_\ell$. The localized Cameron--Martin
inequality
$\Prob_x(A)\ge e^{-\frac12\int g^2-K}\bigl(\widetilde\Prob_x(A)-\widetilde\Prob(\int g\,d\widetilde W^{\mathrm{OU}}>K)\bigr)$
applies with $\tfrac12\int g^2=u_\ell^2\Delta s/8$ and $K:=\ell^{3/2}$ (any exponent
strictly between $1$ and $2$ would do); since
$\int g\,d\widetilde W^{\mathrm{OU}}\sim N(0,u_\ell^2\Delta s/4)$ and $K$ is at least $\ell^{1/2}$
standard deviations, the subtracted term is at most $e^{-\ell/2}$. For the main term,
the recentred process $\widetilde U:=U-u_\ell$ is a centred Ornstein--Uhlenbeck
process started at $x-u_\ell\in[0,1]$. Since $\ell-u_\ell=-\ell\varepsilon_3\le-1$ for
$\ell\ge1/\varepsilon_3$, the event contains
\begin{equation*}
\bigl\{\widetilde U_s\ge-1\ \text{for}\ s\le\Delta s,\ \ \widetilde U_{\Delta s}\in[0,1]\bigr\}.
\end{equation*}
That event
contains the tube event $\{\sup_{s\le\Delta s}|\widetilde U_s-\psi(s)|<\tfrac14\}$ with
$\psi$ linear from $x-u_\ell$ to $\tfrac12$. The law of the centred process on
$[0,\Delta s]$ is a nondegenerate Gaussian measure with full support on continuous paths
from $x-u_\ell$, so the tube probability is positive. After the time change
$s=r\,\Delta s$, $r\in[0,1]$, the law of the process depends weakly continuously on
$(x-u_\ell,\Delta s)$ and the tube becomes an open set of paths, so by the portmanteau
theorem the tube probability is lower semicontinuous on the compact set
$[0,1]\times[\tfrac12,1]$. It therefore attains a positive minimum there, which we call
$2p_1$. Therefore
$\Prob_x(A)\ge e^{-u_\ell^2\Delta s/8-\ell^{3/2}}(2p_1-e^{-\ell/2})\ge p_1e^{-u_\ell^2\Delta s/8-\ell^{3/2}}$
for $\ell\ge\ell_0$.
\end{proof}

Concatenating the entry segment with about $\ln(V_T/v_L)$ such blocks yields the
conditional probability of $E_B$.

\begin{lemma}[Barrier Tracking]\label{lem:B4}
For every fixed $L\ge L_0$ and $\varepsilon_3\in(0,\tfrac14)$, conditionally on a
volatility path in the tube,
\[
\Prob(E_B)\ \ge\ \exp\Bigl(-\bigl(1+\tfrac1L\bigr)(1+\varepsilon_3)^2\,
\frac{\ell_+^2}{8}\,S\,(1+o(1))\Bigr),\qquad
S:=\ln(V_T/v_b)=2n(1+\varepsilon_1)(1+o(1)),
\]
with $o(1)\to0$ as $k\to\infty$ for fixed $(L,\varepsilon_3,\varepsilon_2,\delta,\varepsilon_1)$,
uniformly over volatility paths in the tube.
\end{lemma}
\begin{proof}
\emph{Domination.} For $v\in[v_L,V_T]$ one has $v+v_\alpha\le(1+1/L)v$, hence
\[
\begin{aligned}
\bigl\{B(v)\ge\ell_L\sqrt v\ \ \forall v\in[v_L,V_T]\bigr\}
&\ \subseteq\ \bigl\{B(v)\ge\ell_+\sqrt{v+v_\alpha}\ \ \forall v\in[v_L,V_T]\bigr\}\\
&\ \subseteq\ \bigl\{B\ge(1+\varepsilon_2)C\hat a\ \text{on}\ [v_L,V_T]\bigr\},
\end{aligned}
\]
while on $[v_b,v_L]$ the moving barrier is at most $\ell_+\sqrt{v_L+v_\alpha}=h_L$. In
the coordinate
\begin{equation*}
U(s)=e^{-s/2}B(e^s),\qquad dU=-\tfrac12U\,ds+dW^{\mathrm{OU}},
\end{equation*}
a stationary Ornstein--Uhlenbeck process with covariance $e^{-|s-s'|/2}$, the first
event reads
\begin{equation*}
\bigl\{U(s)\ge\ell_L\ \text{for all } s\in[s_L,s_T]\bigr\}.
\end{equation*}

\emph{Entry and block concatenation.} Apply \cref{lem:entry}. Then, by the Markov
property of $B$ at the $\sigma(a)$-measurable operational time $v_L$ and the Markov
property of $U$ at the endpoints $s_L<s_L+\Delta s<\dots<s_T$ of $N$ blocks of equal
duration, where
\begin{equation*}
N:=\lceil s_T-s_L\rceil,\qquad \Delta s:=\frac{s_T-s_L}{N}\in\bigl(\tfrac12,1\bigr],
\end{equation*}
\cref{lem:block}---uniform over starting points in $I_{\ell_L}$ and over
$\Delta s$---applies at every block:
\[
\Prob\bigl(\text{entry};\ U\ge\ell_L\ \text{on}\ [s_L,s_T];\ U(s_T)\in I_{\ell_L}\bigr)
\ \ge\ e^{-c_L\ell_+^2}\;p_1^N\exp\Bigl(-\frac{u_0^2}{8}(s_T-s_L)-N\ell_L^{3/2}\Bigr).
\]
Since $N\le S+1$, $\ell_L^{3/2}=O(n^{3/4})$ and $\ell_+^2=O(n)$, all terms other
than $u_0^2(s_T-s_L)/8$ total $O(n^{7/4})=o(n^2)$.

\emph{The constructed event implies $E_B$.} On the intersection above,
$B\ge-X_0/(2\orth)$ on $[0,v_b]$ and $B\ge h_L\ge(1+\varepsilon_2)C\hat a$ on $[v_b,v_L]$
(entry); $B\ge(1+\varepsilon_2)C\hat a$ on $[v_L,V_T]$ (domination); and the terminal
threshold holds because
\[
B(V_T)\ \ge\ u_0\sqrt{V_T}\ \ge\ \ell_+\sqrt{V_T+v_\alpha}\;+\;\varepsilon_3\ell_+\sqrt{V_T}\ \ge\ (1+\varepsilon_2)C\,a_T\;+\;\frac{(1+\varepsilon_2)m}{\orth}
\]
for $k\ge k_0$: the first summand dominates $(1+\varepsilon_2)C\hat a(V_T)=(1+\varepsilon_2)Ca_T$
by \eqref{eq:ahat-sqrt}, and
$\varepsilon_3\ell_+\sqrt{V_T}\ge\varepsilon_3(1+\varepsilon_2)C\,a_T\,(1-o(1))\ge(1+\varepsilon_2)m/\orth$
since $a_T\ge e^{-\delta}\alpha e^{n(1+\varepsilon_1)}\gg m$. Hence the constructed event
is contained in $E_B$, and $u_0^2(s_T-s_L)\le(1+1/L)(1+\varepsilon_3)^2\ell_+^2S$
gives the claim. The volatility path enters only through \eqref{eq:ahat-sqrt} and the
bounds $c_{\mathrm b,1}/\theta\le v_b\le c_{\mathrm b,2}/\theta$, which hold for every path
in the tube; the constants and the $o(1)$ term are therefore uniform over the tube.
\end{proof}

\subsection{Assembly}\label{sec:lower-assembly}
The volatility-tube and conditional barrier-tracking costs add at the quadratic scale;
all entry, block-concatenation, localization and bootstrap terms are of lower order.

\begin{proof}[Proof of \cref{prop:lower}]
By \cref{lem:cev,lem:B3}, $P(k)\ge\Prob(A_{\mathrm{vol}}\cap E_B)$, and
by \cref{lem:cond}, \eqref{eq:tube} and \cref{lem:B4}, whose bound is uniform over the
tube,
\[
-\ln P(k)\ \le\ \frac{n^2}{2\nu^2T}\Bigl[(1+\varepsilon_1)^2+\nu^2C^2\bigl(1+\tfrac1L\bigr)(1+\varepsilon_3)^2(1+\varepsilon_2)^2e^{4\delta}(1+\varepsilon_1)^2\Bigr](1+o(1)),
\]
where we used
\begin{equation*}
\Bigl(1+\frac1L\Bigr)(1+\varepsilon_3)^2\,\frac{\ell_+^2S}{8}
=\nu^2C^2\Bigl(1+\frac1L\Bigr)(1+\varepsilon_3)^2(1+\varepsilon_2)^2e^{4\delta}(1+\varepsilon_1)^2\,
\frac{n^2}{2\nu^2 T}\,(1+o(1)).
\end{equation*}
Let $k\to\infty$, then $\varepsilon_3,\varepsilon_1,\varepsilon_2,\delta\downarrow0$ and
$L\to\infty$, and use $1+\nu^2C^2=(1-\rho^2)^{-1}$. The bracket is a product of factors
that each tend to $1$, so after $k\to\infty$ the order of the remaining limits does not
matter.
\end{proof}

\section{Nonnegative Correlation}\label{sec:pos}

For nonnegative correlation the correlation term no longer creates a growing negative
barrier, and the rate is the unconstrained one. The upper bound therefore needs only a
terminal Gaussian constraint, while the lower bound must still control the singular
drift. Throughout this section $\rho\in[0,1)$.

\subsection{Upper Bound}
Even after the survival constraint is dropped, no terminal strategy improves on the
unconstrained quadratic rate.

\begin{proposition}[Upper Bound, $\rho\ge0$]\label{prop:pos-upper}
For $\rho\in[0,1)$,
\[
\liminf_{k\to\infty}\frac{-\ln P(k)}{k^2}\;\ge\;c.
\]
\end{proposition}

The proof drops the drift and the survival constraint (\cref{lem:A1}). A large terminal
value then needs either a large final volatility or a large orthogonal displacement, and
both are controlled through the running maximum of the volatility Brownian motion; no
skeleton is needed.

\begin{proof}[Proof of \cref{prop:pos-upper}]
By \cref{lem:cev,lem:A1}, whose proof uses only the sign of the drift, and dropping the
survival constraint as well, $P(k)\le\Prob(\bar X_T\ge m)$. Because $\rho\ge0$,
\begin{equation*}
\bar X_T=X_0+\frac\rho\nu(a_T-\alpha)+\orth B(V_T)\ \le\ X_0+\frac\rho\nu a_T+\orth B(V_T).
\end{equation*}
Hence $\bar X_T\ge m$ forces $\frac\rho\nu a_T\ge\frac m2$ or
$\orth B(V_T)\ge\frac m2-X_0\ge\frac m4$, the latter for $k$ large. Let
$\bar Z:=\sup_{t\le T}Z_t$, the running maximum of the volatility Brownian motion, so that $a_t\le\alpha e^{\nu\bar Z}$ and $V_T\le T\alpha^2e^{2\nu\bar Z}$. By
the reflection principle,
\begin{equation*}
\Prob(\bar Z\ge s)=2\Prob(Z_T\ge s)\le e^{-s^2/(2T)}\qquad(s\ge0).
\end{equation*}
For $\rho>0$ the first event implies $\nu\bar Z\ge\ln(\nu m/(2\rho\alpha))=n+O(1)$,
so its probability is at most $e^{-(n-O(1))^2/(2\nu^2T)}$; for $\rho=0$ it is empty. For the
second event fix $\varepsilon\in(0,1)$. Conditionally on $\sigma(a)$, $B(V_T)$ is centred
Gaussian with variance $V_T$ (\cref{lem:cond}), so, splitting according to whether
$\nu\bar Z\ge(1-\varepsilon)n$,
\begin{equation*}
\Prob\bigl(\orth B(V_T)\ge\tfrac m4\bigr)\ \le\ \E\exp\Bigl(-\frac{m^2}{32\orth^2V_T}\Bigr)
\ \le\ \Prob\bigl(\nu\bar Z\ge(1-\varepsilon)n\bigr)
+\exp\Bigl(-\frac{m^2e^{-2(1-\varepsilon)n}}{32\orth^2T\alpha^2}\Bigr).
\end{equation*}
The first term is at most $e^{-(1-\varepsilon)^2n^2/(2\nu^2T)}$. Since
$m^2=e^{2n}f_0^{2(1-\beta)}/(1-\beta)^2$, the second is $\exp(-c_1e^{2\varepsilon n})$ for a
constant $c_1>0$.
Altogether $P(k)\le3e^{-(1-\varepsilon)^2n^2/(2\nu^2T)}$ for all large $k$. Since
$n^2/(2\nu^2T)=ck^2$, letting $k\to\infty$ and then $\varepsilon\downarrow0$ proves the
proposition.
\end{proof}

\subsection{Lower Bound}\label{sec:pos-lower}
For the matching lower bound it suffices that survival and the terminal adjustment cost
only $O(n)$ beyond the volatility excursion.

\begin{proposition}[Lower Bound, $\rho\ge0$]\label{prop:pos-lower}
For $\rho\in[0,1)$,
\[
\limsup_{k\to\infty}\frac{-\ln P(k)}{k^2}\;\le\;c.
\]
\end{proposition}

The proof has three steps. Volatility is confined to an exponential tube, at cost
$ck^2+O(n)$ (\cref{lem:tube}). On the tube the forward, in operational time, dominates a
Bessel process killed at zero (\cref{lem:comparison}). The explicit transition density of
that process shows that survival and the terminal adjustment cost only $O(n)$
(\cref{lem:bessel}). The bootstrap of
\cref{sec:lower} is not available here: there the barrier forces $\bar X$ to be of order
$a$, which makes the singular drift negligible, whereas for $\rho\ge0$ nothing pushes
$\bar X$ up, and $\rho=0$ is exactly this case.

\emph{Volatility tube.} Fix $\delta\in(0,\tfrac14)$ with $\rho\alpha\delta/\nu\le X_0/4$, and
apply \cref{lem:tube} with $\theta:=n/(\nu T)$, so that $a^*(t)=\alpha e^{nt/T}$ and
\begin{equation}\label{eq:pos-tube}
\Prob(A_{\mathrm{vol}})\ge\exp\bigl(-n^2/(2\nu^2T)-O(n)\bigr).
\end{equation}
On the tube $a_t\ge\alpha e^{-\delta}$; hence, as $\rho\ge0$,
\begin{equation}\label{eq:pos-xi}
\xi_t=\tfrac\rho\nu(a_t-\alpha)\ge-\tfrac{\rho\alpha}{\nu}(1-e^{-\delta})\ge-\tfrac{\rho\alpha\delta}{\nu}\ge-\tfrac{X_0}{4}.
\end{equation}
Moreover, \cref{lem:tube} and $V^*(T)=\alpha^2T(e^{2n}-1)/(2n)$ give, on the tube,
\begin{equation}\label{eq:pos-VT}
e^{-2\delta}\,\frac{\alpha^2T(e^{2n}-1)}{2n}\ \le\ V_T\ \le\ e^{2\delta}\,\frac{\alpha^2Te^{2n}}{2n}.
\end{equation}

\emph{Comparison process.} Let $Y^-$ denote a solution of
\begin{equation}\label{eq:Yminus}
dY^-_v=\orth\,dB_v-\frac{b}{Y^-_v}\,dv,
\end{equation}
absorbed at its first zero. With $\mu:=\tfrac12+b/\orth^2$, the process $Y^-/\orth$ solves the
Bessel equation of dimension $2-2\mu<1$, so its survival can be computed exactly. The
following lemma reduces survival of $X$ to survival of $Y^-$ on the operational-time
horizon.

\begin{lemma}[Comparison on the Operational-Time Horizon]\label{lem:comparison}
Fix a volatility path in the tube, and put $Y_v:=X_{t(v)}$ and $\hat\xi_v:=\xi_{t(v)}$, so
that
\[
Y_v=X_0+\hat\xi_v+\orth B_v-b\int_0^v\frac{du}{Y_u},\qquad 0\le v\le V_T,
\]
up to the first zero of $Y$. Let $Y^-$ solve \eqref{eq:Yminus} on $[0,V_T]$ with
$Y^-_0=X_0/2$ and the same $B$, and let $\zeta^-:=\inf\{v\in[0,V_T]:Y^-_v=0\}$. Then
$Y_v\ge Y^-_v$ for $0\le v\le V_T\wedge\zeta^-$. In particular, if $\zeta^-=\infty$ then
$\zeta>T$ and $X_T=Y_{V_T}\ge Y^-_{V_T}$.
\end{lemma}
\begin{proof}
The path $\hat\xi$ is continuous but need not have finite variation; only the bound
$\hat\xi\ge-X_0/4$ from \eqref{eq:pos-xi} is used. For $v<\zeta^-$ and before any first
crossing of $G_v:=Y_v-Y^-_v$,
\[
G_v=\frac{X_0}{2}+\hat\xi_v+b\int_0^v\Bigl(\frac1{Y^-_u}-\frac1{Y_u}\Bigr)du\ \ge\ \frac{X_0}{4}>0.
\]
By continuity no crossing occurs, so $Y\ge Y^-+X_0/4>0$ on $[0,V_T\wedge\zeta^-)$, and the
claims follow.
\end{proof}

The killed Bessel process has an explicit transition density, which gives the cost of
survival and of the terminal adjustment directly.

\begin{lemma}[Killed Bessel Lower Bound]\label{lem:bessel}
Let $Y^-$ solve \eqref{eq:Yminus} with $Y^-_0=y>0$, absorbed at its first zero, and let
$\mu=\tfrac12+b/\orth^2$. Then for all $t,z>0$,
\begin{equation*}
\Prob\bigl(Y^-_v>0\ \text{for all}\ v\le t,\ \ Y^-_t\ge z\bigr)\ \ge\
\frac{1}{\Gamma(\mu+1)}\Bigl(\frac{y^2}{2\orth^2t}\Bigr)^{\mu}\exp\Bigl(-\frac{y^2+z^2}{2\orth^2t}\Bigr).
\end{equation*}
\end{lemma}
\begin{proof}
Put $\hat Y:=Y^-/\orth$ and $x:=y/\orth$. Up to its first zero, $\hat Y$ solves
$d\hat Y=dB-(\mu-\tfrac12)\,dv/\hat Y$, and the function $w\mapsto w^{2\mu}$ is annihilated by the
generator $\tfrac12\partial_w^2-(\mu-\tfrac12)w^{-1}\partial_w$. For $\varepsilon<x<1/\varepsilon$
let $\tau_\varepsilon$ be the exit time of $\hat Y$ from $[\varepsilon,1/\varepsilon]$. By It\^o's
formula, $(\hat Y_{v\wedge\tau_\varepsilon}/x)^{2\mu}$ is a bounded positive martingale with mean
$1$, and by Girsanov's theorem it is the density of a probability measure under which $\hat Y$
solves $d\hat Y=d\widetilde B+(\mu+\tfrac12)\,dv/\hat Y$ up to $\tau_\varepsilon$, the equation of a Bessel
process $\check Y$ of dimension $2+2\mu$ started at $x$. Hence, for measurable $f\ge0$,
\begin{equation*}
\E\bigl[f(\hat Y_t);\ t<\tau_\varepsilon\bigr]
=x^{2\mu}\,\E\bigl[f(\check Y_t)\,\check Y_t^{\,-2\mu};\ t<\check\tau_\varepsilon\bigr],
\end{equation*}
where $\check\tau_\varepsilon$ is the exit time of $\check Y$ from $[\varepsilon,1/\varepsilon]$.
As $\varepsilon\downarrow0$, the event $\{t<\tau_\varepsilon\}$ increases to the event that $\hat Y$
has no zero on $[0,t]$, because $\hat Y\le x+B$ cannot explode; and $\{t<\check\tau_\varepsilon\}$
increases to the whole space, because a Bessel process of dimension $2+2\mu>2$ neither hits
$0$ nor explodes \cite[Ch.~XI, \S1]{RY}. By monotone convergence and the transition density
of $\check Y$ \cite[Ch.~XI, \S1]{RY},
\begin{equation*}
\Prob\Bigl(\hat Y_v>0\ \text{for all}\ v\le t,\ \hat Y_t\ge\frac z\orth\Bigr)
=\int_{z/\orth}^\infty\frac wt\Bigl(\frac xw\Bigr)^{\mu}e^{-\frac{x^2+w^2}{2t}}
I_\mu\Bigl(\frac{xw}{t}\Bigr)\,dw,
\end{equation*}
where $I_\mu$ is the modified Bessel function of the first kind. Its power series has
nonnegative terms, so $I_\mu(r)\ge(r/2)^\mu/\Gamma(\mu+1)$ for $r\ge0$, and the integrand is
at least
\begin{equation*}
\frac{x^{2\mu}}{\Gamma(\mu+1)\,(2t)^{\mu}}\cdot\frac wt\,e^{-(x^2+w^2)/(2t)}.
\end{equation*}
Integrating in $w$ over $[z/\orth,\infty)$ gives the claim.
\end{proof}

The two lemmas combine with the volatility tube to give the lower bound.

\begin{proof}[Proof of \cref{prop:pos-lower}]
Fix a volatility path in the tube. By \cref{lem:cev,lem:comparison}, the event
$\{F_T\ge f_0e^k\}$ contains the event that the solution $Y^-$ of \eqref{eq:Yminus} with
$Y^-_0=X_0/2$ has no zero on $[0,V_T]$ and satisfies $Y^-_{V_T}\ge m$. Conditionally on
$\sigma(a)$, $B$ is a standard Brownian motion and $V_T$ is fixed (\cref{lem:cond}), so by
\cref{lem:bessel} with $y=X_0/2$, $z=m$ and $t=V_T$ the conditional probability of this event
is at least
\begin{equation*}
\frac1{\Gamma(\mu+1)}\Bigl(\frac{X_0^2}{8\orth^2V_T}\Bigr)^{\mu}
\exp\Bigl(-\frac{X_0^2/4+m^2}{2\orth^2V_T}\Bigr).
\end{equation*}
By \eqref{eq:pos-VT} and $m^2=e^{2n}f_0^{2(1-\beta)}/(1-\beta)^2$, uniformly over the tube
$\ln V_T=2n+O(\ln n)$ and $m^2/V_T=O(n)$, so this bound is $e^{-O(n)}$. Integrating over the
tube with \eqref{eq:pos-tube} gives
\[
-\ln P(k)\ \le\ \frac{n^2}{2\nu^2T}+O(n),
\]
which proves the proposition.
\end{proof}

\subsection{Proof of \texorpdfstring{\cref{thm:main-tail}}{Theorem 1.1}}
The four propositions combine to give the main theorem and its corollaries.

\begin{proof}[Proof of \cref{thm:main-tail}]
For $\rho<0$ the theorem is \cref{prop:upper,prop:lower}; for
$\rho\ge0$ it is \cref{prop:pos-upper,prop:pos-lower}. The
lower-bound constructions show $P(k)>0$ for all large $k$, so \cref{lem:transfer}
applies and yields \cref{cor:wing,cor:resolution}. At $\rho=0$ this
recovers \cite{BFL}.
\end{proof}

\section{The Intrinsic Distance in the Survival Domain}\label{sec:geometry}

We now prove the geometric identification \eqref{eq:intro-c2}. Henry-Labord\`ere's
derivation of (10.22) in \cite{HL} represents the SABR generator by a Riemannian metric,
a connection and a scalar potential on a patch of the Poincar\'e half-plane, and then
uses a Gaussian estimate based on the unrestricted geodesic distance. For the absorbed
model the leading quantity turns out to be the intrinsic distance within the open
survival domain, the state space of the process killed at the boundary. That killed
diffusions are governed by the intrinsic distance of their domain is classical in small
time, through heat-kernel asymptotics in terms of the distance function \cite{Norris} and
the principle of not feeling the boundary \cite{Hsu,CY}; here the same quantity appears at
fixed maturity and extreme strikes. \Crefrange{sec:fermi}{sec:dfree} prove the
identification; \cref{sec:potential} and \cref{rem:heat-kernel} return to Henry-Labord\`ere's
potential and to what a heat-kernel proof would require.

\subsection{Fermi Coordinates and the Survival Domain}\label{sec:fermi}
The hyperbolic geometry of SABR is most transparent in Fermi coordinates around a
geodesic, in which the survival domain becomes the coordinate half-plane $\{R>\Rb\}$,
bounded by a curve at constant distance from that geodesic. Recall
$X=F^{1-\beta}/(1-\beta)$, and set
\[
y:=\frac a\nu,\qquad x:=\frac{X-\rho y}{\orth},\qquad
R:=\operatorname{arsinh}(x/y),\qquad q:=\tfrac12\ln(x^2+y^2),
\qquad \Rb:=-\operatorname{artanh}\rho.
\]
Before absorption, $x$ and $y$ are driven by the independent Brownian motions $Z^\perp$
and $Z$ with the common volatility $a=\nu y$, so the principal symbol of the generator is
that of the hyperbolic metric $(dx^2+dy^2)/(\nu y)^2$. In the coordinates $(R,q)$,
\[
x=e^q\tanh R,\qquad y=e^q\operatorname{sech}R,
\qquad X=e^q\Psi(R),\qquad
\Psi(R):=\orth\tanh R+\rho\operatorname{sech}R,
\]
and the open survival domain $\{X>0\}$ and the metric are
\begin{equation}\label{eq:killed-fermi-geometry}
D=\{R>\Rb\},
\qquad
g=\nu^{-2}\bigl(dR^2+\cosh^2R\,dq^2\bigr).
\end{equation}
Thus $R/\nu$ is the signed distance from the geodesic $\{x=0\}$, and motion in $q$ at
fixed $R$ costs $\cosh R/\nu$ per unit of $q$ (\cref{fig:geometry}).

The survival domain is the moving barrier of \cref{sec:mech} in these coordinates. Before
absorption, \eqref{eq:decomp} gives $x=M+(X_0-\rho\alpha/\nu-\mathcal I)/\orth$, where
$\mathcal I:=b\int_0^\cdot a^2/X\,ds$ is the singular drift, so for $\rho<0$ the survival
condition $X=\orth x+\rho y>0$ reads $x>Ca$: apart from the initial terms and the drift,
this is the barrier $B(V)\ge Ca$, and $R>\Rb$ is its form in Fermi coordinates. At fixed
$R$ one has $q=\ln a+O(1)$, so the heuristic optimal path of \cref{sec:mech}, exponential
volatility with $B$ tracking the barrier, is the segment $R\approx\Rb$ of
\cref{fig:geometry}(b). Its cost per unit of $q$, $\cosh\Rb/\nu=1/(\nu\orth)$, is the
reciprocal of the effective volatility of volatility $\nu\orth$ found in \cref{sec:mech}.

Distances in $D$ are measured along paths that stay in $D$. For $z,z'\in D$, let
$\mathcal A_D(z,z')$ be the class of finite chains
$\gamma=(\gamma_1,\ldots,\gamma_N)$, $N\ge1$, of $C^1$ legs
$\gamma_j=(R_j,q_j):[0,1]\to D$ with $\gamma_1(0)=z$, $\gamma_N(1)=z'$ and
$\gamma_j(1)=\gamma_{j+1}(0)$. Every leg respects the domain constraint, while minimizing
sequences may approach the boundary. For $\gamma\in\mathcal A_D(z,z')$ set
\[
L_g(\gamma):=\frac1\nu\sum_{j=1}^N\int_0^1
\sqrt{\dot R_j(s)^2+\cosh^2(R_j(s))\,\dot q_j(s)^2}\,ds,
\qquad
d_D(z,z'):=\inf_{\gamma\in\mathcal A_D(z,z')}L_g(\gamma),
\]
and $d_D(z,A):=\inf_{z'\in A}d_D(z,z')$ for $A\subset D$. Let $z_0=(R_0,q_0)$ be the
initial point, so that $X_0=e^{q_0}\Psi(R_0)$ and $m=X_0e^n$, and let
\[
\mathcal T_k:=\{(R,q)\in D:e^q\Psi(R)\ge m\}
\]
be the terminal target set.

\subsection{The Distance to the Target Set}
Reaching $\mathcal T_k$ requires a displacement of about $n$ in $q$, and such motion is
cheapest where $\cosh R$ is smallest. The distance is therefore governed by the infimum
of $\cosh R$ over $D$.

\begin{proposition}[Intrinsic Distance to the Target Set]\label{prop:killed-distance}
As $k\to\infty$, with $s_\rho$ as in \cref{sec:results} ($s_\rho=1$ for $\rho\ge0$ and
$s_\rho=\orth$ for $\rho<0$),
\begin{equation}\label{eq:killed-distance}
d_D(z_0,\mathcal T_k)
=\frac{n}{\nu s_\rho}+O(1)
=\frac{(1-\beta)k}{\nu s_\rho}+O(1).
\end{equation}
For the distance $d_{\mathrm{free}}$ in the same metric, with chains allowed to range over the
whole $(R,q)$-plane, $d_{\mathrm{free}}(z_0,\mathcal T_k)=n/\nu+O(1)$ for every $\rho$.
\end{proposition}
\begin{proof}
Since $\Psi(R)=(\orth\sinh R+\rho)/\cosh R$, one has $\Psi>0$ precisely when $R>\Rb$.
As $(\orth,\rho)$ and $(\tanh R,\operatorname{sech}R)$ are unit vectors, $0<\Psi\le1$ on
$D$. Hence every $(R,q)\in\mathcal T_k$ satisfies
\[
q-q_0\ge n+\ln\Psi(R_0)-\ln\Psi(R)\ge n+\ln\Psi(R_0).
\]
Moreover $\inf_{R>\Rb}\cosh R=1/s_\rho$: if $\rho\ge0$, then $\Rb\le0$ and the infimum is
$\cosh0=1$ (approached from within $D$ when $\rho=0$); if $\rho<0$, then $\Rb>0$ and the
infimum is $\cosh\Rb=1/\orth$. Therefore every admissible chain from $z_0$ to
$\mathcal T_k$ satisfies, for all sufficiently large $k$,
\[
L_g(\gamma)\ge \frac1{\nu s_\rho}\sum_{j=1}^N\int_0^1|\dot q_j(s)|\,ds\ge \frac{n+\ln\Psi(R_0)}{\nu s_\rho}.
\]
For the reverse bound, first let $\rho>0$. The three-segment path
\[
(R_0,q_0)\longrightarrow(0,q_0)\longrightarrow(0,q_0+n)\longrightarrow(R_0,q_0+n)
\]
stays in $D$, ends in $\mathcal T_k$, and has length $n/\nu+2|R_0|/\nu$
(\cref{fig:geometry}(a)). If $\rho\le0$,
set $R_k:=\Rb+n^{-1}$. For all sufficiently large $k$ the analogous path through $R_k$
(\cref{fig:geometry}(b)) stays in $D$, has the same endpoint $(R_0,q_0+n)\in\mathcal T_k$, and has length
\[
\frac{2(R_0-R_k)}{\nu}+\frac{n\cosh R_k}{\nu}
=\frac{n\cosh\Rb}{\nu}+O(1)
=\frac{n}{\nu s_\rho}+O(1).
\]
Thus for $\rho\le0$ the bound is approached by paths whose long tangential segment tends
to the absorbing boundary, whereas for $\rho>0$ the minimizing tangential line $R=0$ lies
in the interior of $D$. Returning to the fixed level $R_0$ makes the terminal constraint exact and avoids the
displacement $-\ln\Psi(R_k)$ that would arise if the endpoint stayed in the boundary
layer. For $d_{\mathrm{free}}$ the same argument applies with $\inf_R\cosh R=1$ over the whole
plane, and the three-segment path through $R=0$ gives the upper bound for every $\rho$,
since it need not stay in $D$.
\end{proof}

\begin{figure}[t]
\centering
\includegraphics[width=\textwidth]{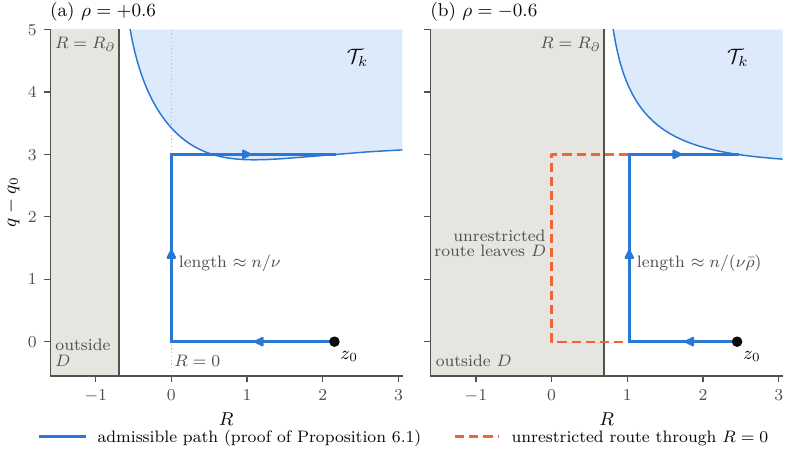}
\caption{Fermi coordinates for the reference parameters of \cref{sec:num} with
$\beta=0.5$ and $n=3$. The survival domain is $D=\{R>R_\partial\}$; the shaded strip lies
outside $D$, and $\mathcal T_k$ is the target set. Motion in $q$ costs $\cosh R/\nu$ per
unit, least on the line $R=0$. (a)~For $\rho>0$ the line $R=0$ lies in $D$, and the paths
of the proof of \cref{prop:killed-distance} have length $n/\nu+O(1)$, as in the
unrestricted metric. (b)~For $\rho<0$ the line $R=0$ lies outside $D$; admissible paths
can only approach the boundary, where the cost is $\cosh R_\partial/\nu=1/(\nu\bar\rho)$,
so the distance grows like $n/(\nu\bar\rho)$.}
\label{fig:geometry}
\end{figure}

Combined with \cref{thm:main-tail}, the proposition identifies the tail rate with
the intrinsic distance. In the language of \cref{conj:HL2}, the unrestricted distance in
the Gaussian estimate is replaced by the intrinsic one; for SABR this is a consequence of
\cref{thm:main-tail} and the explicit computation above.

\begin{corollary}[Intrinsic-Distance Form of the Tail Rate]\label{cor:corrected-c2}
For every $\rho\in(-1,1)$,
\begin{equation}\label{eq:corrected-c2}
-\ln P(k)=\frac{d_D(z_0,\mathcal T_k)^2}{2T}+o(k^2).
\end{equation}
\end{corollary}
\begin{proof}
By \cref{prop:killed-distance},
$d_D(z_0,\mathcal T_k)^2/(2T)=(1-\beta)^2k^2/(2\nu^2Ts_\rho^2)+o(k^2)$, which equals
$-\ln P(k)+o(k^2)$ by \cref{thm:main-tail}.
\end{proof}

\subsection{Unrestricted versus Intrinsic Distance}\label{sec:dfree}
The comparison with the unrestricted distance shows exactly where the Gaussian estimate
behind \cref{conj:HL2} breaks down. By \cref{prop:killed-distance}, $d_{\mathrm{free}}(z_0,\mathcal T_k)=n/\nu+O(1)$ for every
$\rho$: the unrestricted plane always permits the long segment on $R=0$. If
$\rho<0$, however, $\Rb>0$ and the line $R=0$ lies outside $D$; admissible curves can
only approach $R=\Rb$, where the tangential cost is $\cosh\Rb/\nu=1/(\nu\orth)$
(\cref{fig:geometry}). If $\rho>0$ the line $R=0$ is interior, and if $\rho=0$ it
is the boundary but can be approached without changing the leading coefficient. Thus the
intrinsic and unrestricted distances agree to leading order exactly when the excluded
region is irrelevant on the $k^2$ scale. Paulot \cite{Paulot} pointed out the
corresponding boundary restriction on small-time geodesic expansions;
\cref{prop:killed-distance} is its fixed-maturity, extreme-strike analogue.

\subsection{The Potential and the Connection}\label{sec:potential}
The endpoint value in \cref{conj:HL2} does not control the lower-order part of the
operator. In the coordinates $(X,a)$ the generator is
$\frac12g^{ij}\partial_i\partial_j-b(a^2/X)\partial_X$, with inverse metric
$(g^{ij})=a^2\bigl(\begin{smallmatrix}1&\rho\nu\\ \rho\nu&\nu^2\end{smallmatrix}\bigr)$.
Since $g^{1/2}g^{ij}$ is constant, the Laplace--Beltrami operator has no first-order part,
and Henry-Labord\`ere's form of the generator, which here reads
$\frac12g^{ij}(\partial_i+A_i)(\partial_j+A_j)+Q$, holds with $A_i=g_{ij}w^j$, where $w=-(ba^2/X)\,\partial_X$ is the drift, and
$Q=-\frac12g^{ij}(\partial_iA_j+A_iA_j)$. This gives the connection one-form
\[
A=-\frac{b}{\orth^2X}\,dX+\frac{\rho b}{\nu\orth^2X}\,da,
\qquad
dA=-\frac{\rho b}{\nu\orth^2X^2}\,dX\wedge da,
\]
which is singular and, for $\rho\ne0$, not globally gauge-removable, and the scalar
potential
\[
Q(F,a)=-\Gamma_{\beta,\rho}\left(\frac aX\right)^2,
\qquad
\Gamma_{\beta,\rho}=\frac b2+\frac{b^2}{2\orth^2}
=\frac{\beta(1-\beta)+\beta^2/(2\orth^2)}{4(1-\beta)^2}>0,
\]
which has an inverse-square singularity at the absorbing boundary. For the unrestricted
saddle on the fibre $\{F=f\}$, with $X=f^{1-\beta}/(1-\beta)$ and
$X_0=f_0^{1-\beta}/(1-\beta)$, the exact formula \cite[(6.45)]{HL} is
\[
a_{\min}(f)^2=\alpha^2+2\rho\alpha\nu(X-X_0)+\nu^2(X-X_0)^2.
\]
Thus $a_{\min}(f)/X\to\nu$ as $f\to\infty$, and the potential at the saddle has the
large-strike limit
\begin{equation}\label{eq:Q}
\lim_{f\to\infty}Q(f,a_{\min}(f))
\;=\; \frac{\nu^2}{4(1-\beta)^2}\Bigl(\beta(\beta-1)-\frac{\beta^2}{2(1-\rho^2)}\Bigr),
\end{equation}
the value given in \cite[p.~306]{HL}. The book states it as a constant; the formulas
above show that it is the large-strike limit, so the potential remains bounded along the
far-right saddle, although its finite-strike value is generally not constant. The
potential and the connection do not affect \cref{cor:corrected-c2}, but the formal
calculation of \cref{sec:fiber} indicates that they contribute at order $k^{3/2}$. We
expect them to be the main obstacle to a heat-kernel proof of \cref{cor:corrected-c2}.

\begin{remark}[A Heat-Kernel Route]\label{rem:heat-kernel}
A Dirichlet heat-kernel theorem would give a second route to \eqref{eq:corrected-c2},
closer to Henry-Labord\`ere's method. Let $K_D$ be the Dirichlet heat kernel of the
generator on $D$---the transition density of the killed process---with respect to a
reference measure $\varpi$, so that $P(k)=\int_{\mathcal T_k}K_D(T;z_0,z)\,\varpi(dz)$. It
would suffice to prove
\[
\ln K_D(T;z_0,z)=-\frac{d_D(z_0,z)^2}{2T}+o(k^2)
\]
uniformly on a subset $\mathcal N_k\subset\mathcal T_k$ with
\begin{equation*}
\sup_{z\in\mathcal N_k}d_D(z_0,z)^2=d_D(z_0,\mathcal T_k)^2+o(k^2),\qquad
-\ln\varpi(\mathcal N_k)=o(k^2),
\end{equation*}
together with an exponentially matching bound for the
integral over $\mathcal T_k\setminus\mathcal N_k$, for instance a global upper estimate
with subquadratic logarithmic volume factors. Gaussian bounds of Li--Yau type
\cite{LiYau,Davies,Sturm} do not apply directly, because the potential and the connection
are singular at $\partial D$; whether a suitable singular-potential framework applies is
open. The proofs in this paper establish both tail inequalities directly.
\end{remark}

\section{Numerical Illustrations}\label{sec:num}

The theorems are asymptotic, so computations at finite strikes can illustrate them but
not test them. We use them for the comparison that matters in practice: the model's right
wing against Hagan's formula. All computations use $f_0=1$, $\alpha=0.3$, $\nu=0.6$,
$T=1$ and $\beta=0.5$, with $\rho\in\{-0.6,0\}$; for $\rho=-0.6$ the two rates of
\cref{thm:main-tail} are $c=0.3472$ and $c_-=0.5425$. The methods are described in
\cref{app:num}.

\begin{table}[tbp]\centering\small
\begin{tabular}{r ccc ccc}
\toprule
 & \multicolumn{3}{c}{$\rho=-0.6$} & \multicolumn{3}{c}{$\rho=0$} \\
\cmidrule(lr){2-4}\cmidrule(lr){5-7}
$k$ & model & Hagan--Ob{\l}{\'o}j & ratio & model & Hagan--Ob{\l}{\'o}j & ratio \\
\midrule
1 & $0.255$ & $0.258$ & $0.989$ & $0.360$ & $0.368$ & $0.980$ \\
2 & $0.334$ & $0.350$ & $0.956$ & $0.457$ & $0.471$ & $0.971$ \\
3 & $0.398$ & $0.432$ & $0.921$ & $0.538$ & $0.557$ & $0.966$ \\
4 & $0.447$ & $0.502$ & $0.890$ & $0.605$ & $0.628$ & $0.963$ \\
5 & $0.484$ & $0.561$ & $0.863$ & $0.660$ & $0.688$ & $0.960$ \\
\midrule
$\infty$ & $0.960$ & $1.217$ & $0.789$ & $1.200$ & $1.236$ & $0.971$ \\
\bottomrule
\end{tabular}
\caption{Right-wing implied volatilities of the model (Monte Carlo; standard errors at most
$0.0003$) and of Hagan's formula in Ob{\l}{\'o}j's form, and their ratio. The last row
gives the limits as $k\to\infty$: \cref{cor:wing} for the model, and the limit stated
after \cref{fig:wing} for the formula.}\label{tab:hagan}
\end{table}

\begin{figure}[tbp]
\centering
\includegraphics[width=\textwidth]{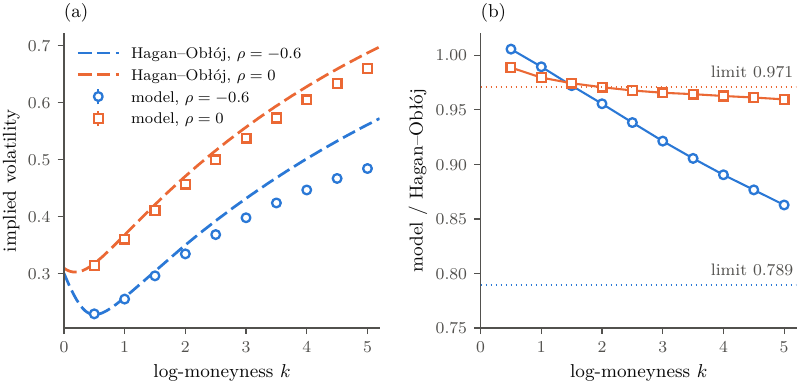}
\caption{(a)~Right-wing implied volatility of the model (markers, Monte Carlo; the error bars
are smaller than the markers) and of Hagan's formula in Ob{\l}{\'o}j's form (dashed), for
$\rho=-0.6$ and $\rho=0$. (b)~Their ratio, with its limits as $k\to\infty$ (dotted).}
\label{fig:hagan}
\end{figure}

\Cref{tab:hagan} and \cref{fig:hagan} show that Hagan's formula overstates the model's
right wing at every strike with $k\ge1$, and that the excess grows much faster for
negative correlation. At $k=4$ the model's implied volatility is $0.447$ against
$0.502$ for $\rho=-0.6$, a ratio of $0.890$, while for $\rho=0$ the ratio is
$0.963$. The theorem determines where the ratios go: by \cref{cor:wing} and the limit
of the formula, they tend to
\begin{equation*}
\frac{\sqrt{1-\rho^2}}{1+(2-3\rho^2)\nu^2T/24}=0.789\quad(\rho=-0.6),
\qquad
\frac{1}{1+\nu^2T/12}=0.971\quad(\rho=0).
\end{equation*}
For $\rho=0$ the ratio stays within $0.012$ of its limit for $1\le k\le5$. For
$\rho=-0.6$ it falls steadily, to $0.863$ at $k=5$: over the strikes shown the gap
between the formula and the model keeps widening, toward the factor that
\cref{cor:wing} predicts.

\section{Discussion}\label{sec:disc}

This section offers two complementary perspectives on the main theorem: what lies beyond the
leading order when $\rho<0$, and how the absorbed model sits between the neighbouring
regimes $\beta=1$ and $\beta=0$.

\subsection{Beyond the Leading Order}\label{sec:fiber}
\cref{thm:main-tail} settles only the leading $k^2$ scale. For $\rho<0$ four
statements of different status must be kept apart: the square-root-boundary eigenvalue
expansion below is classical for the drift-free comparison process; the Fourier--Laplace
reduction of the full generator with Dirichlet boundary condition is exact; the ensuing
boundary-layer localization and Legendre inversion are formal; and the resulting
$k^{3/2}$ correction is conjectural.

Along the leading exponential volatility excursion, dropping the singular drift from the
survival constraint gives the moving-barrier comparison problem $B(v)\ge\ell\sqrt v$.
Under $v=e^s$, $U(s)=e^{-s/2}B(e^s)$ is a stationary Ornstein--Uhlenbeck process
constrained to stay above $\ell$, with survival exponent given by the principal Dirichlet
eigenvalue $\Theta(\ell)$ on $(\ell,\infty)$; this is the square-root-boundary theory of
Breiman \cite{Breiman}, Novikov \cite{Novikov} and Uchiyama \cite{Uchiyama}. Its classical large-$\ell$ expansion
is
\begin{equation}\label{eq:theta-exp}
\Theta(\ell)=\frac{\ell^2}{8}+\frac{2^{1/3}|a_1|}{4}\,\ell^{2/3}-\frac14+o(1),
\qquad |a_1|=2.33811\ldots,
\end{equation}
where $a_1$ is the first zero of the Airy function. The leading term reproduces $c_-$.
The Airy term is valid for the comparison problem, but it cannot be transferred to the
full SABR diffusion without restoring the singular drift.

The exact lower-order obstruction appears after a transform in the tangential variable.
In the coordinates $(R,q)$ of \cref{sec:geometry}, with
\[
\chi(R)=\frac{1}{\orth\cosh R\,(\orth\sinh R+\rho)},
\]
the backward generator on $D$, with Dirichlet boundary condition at $R=\Rb$, is
\[
\mathcal L
=\frac{\nu^2}{2}\left(\partial_R^2+\tanh R\,\partial_R+\operatorname{sech}^2R\,\partial_q^2\right)
-b\nu^2\chi(R)\left(\partial_R+\tanh R\,\partial_q\right).
\]
All coefficients and the boundary are invariant under translation in $q$, so $\mathcal L$
acts on each mode in $q$ through a one-dimensional fibre operator. A real Fourier mode
yields a non-self-adjoint fibre. Positive $q$-displacement leads instead to the bilateral
Laplace mode $e^{\lambda q}$, $\lambda>0$, equivalently a Fourier mode with imaginary
frequency $-i\lambda$; we write $\mathcal L(u(R)e^{\lambda q})=e^{\lambda q}\mathcal L_\lambda u$.
After conjugation by the radial speed measure, this analytically continued fibre
$\mathcal L_\lambda$ is a real one-dimensional differential expression, and we write
$H_\lambda$ for its negative, in the Friedrichs realization that carries the Dirichlet
condition. For $\rho<0$ and $r:=R-\Rb$, the expansion of $H_\lambda$ at the boundary,
shifted by $\nu^2\orth^2\lambda^2/2=\max_D\frac{\nu^2}{2}\lambda^2\operatorname{sech}^2R$, is
\begin{equation}\label{eq:fiber-layer}
H_\lambda+\frac{\nu^2\orth^2\lambda^2}{2}
=\nu^2\left[-\frac12\partial_r^2
+|\rho|\orth^2\lambda^2r+\frac{b|\rho|\lambda}{r}
+\frac{b(b+1)}{2r^2}
+O(\lambda^2r^2+\lambda+r^{-1}+1)\right].
\end{equation}
The term $b|\rho|\lambda/r$ is the singular tangential component of the drift. Balancing it
with $|\rho|\orth^2\lambda^2r$ gives
\[
r_\lambda=\frac{\sqrt b}{\orth\sqrt\lambda},
\qquad
|\rho|\orth^2\lambda^2r_\lambda+\frac{b|\rho|\lambda}{r_\lambda}
=2|\rho|\orth\sqrt b\,\lambda^{3/2}.
\]
A formal one-dimensional min--max localization, with a bump of width $\lambda^{-5/8}$
about $r_\lambda$, predicts for the principal eigenvalue
$\Theta_{\mathrm{fib}}(\lambda):=-\inf\operatorname{spec}H_\lambda$ of the tilted fibre the
scaling
\begin{equation}\label{eq:fiber-eigen}
\Theta_{\mathrm{fib}}(\lambda)
=\frac{\nu^2\orth^2}{2}\lambda^2
-2\nu^2|\rho|\orth\sqrt b\,\lambda^{3/2}
+O(\lambda^{5/4}).
\end{equation}
We do not prove \eqref{eq:fiber-eigen}: a proof must specify the quadratic-form domain,
establish the relevant discrete eigenvalue, and supply a global localization upper
bound. The inverse-square term is only $O(\lambda)$ at this scale but formally selects the
Dirichlet Friedrichs realization. At the ordinary Airy scale $r\asymp\lambda^{-2/3}$ the
term $b|\rho|\lambda/r$ is of order $\lambda^{5/3}$ and dominates the $\lambda^{4/3}$ Airy
balance; within this formal balance the Airy regime survives only in the driftless case
$b=0$, or for the comparison process without the singular drift.

Formal Legendre inversion of \eqref{eq:fiber-eigen}, and a second, likewise formal,
calculation that conditions on the autonomous angular path and optimizes its exact
Gaussian tail, give the same prediction.

\begin{conjecture}[Singular-Drift Subleading Correction]\label{conj:subleading}
For $\beta\in(0,1)$ and $\rho\in(-1,0)$,
\[
-\ln P(k)
=c_-k^2+c_{3/2}k^{3/2}+o(k^{3/2}),
\qquad
c_{3/2}
=\frac{\sqrt{2\beta}(1-\beta)|\rho|}{\nu(1-\rho^2)\sqrt T}.
\]
\end{conjecture}
A proof would require a uniform shrinking-tube lower bound or, alternatively, complex
saddle-point estimates controlling endpoint overlaps and transverse frequencies.
Computations at moderate strikes, such as those of \cref{sec:num}, cannot distinguish the
$k^{3/2}$ term reliably from other $o(k^2)$ corrections.

\subsection{Neighbouring Parameter Regimes}\label{sec:regimes}
The geometry above identifies absorption, rather than correlation alone, as the source of
the negative-correlation correction. The neighbouring regimes are therefore informative:
the correlation structure persists, but the state space, and with it the tail mechanism,
changes. \Cref{tab:lit} lists the results on extreme strikes and boundary effects for
SABR-type models that we located, with their scope; smile limits of this kind are also
known for other models, for instance the large-maturity Heston smile \cite{FJ}.

\begin{table}[tb]\centering\small
\begin{tabular}{>{\raggedright\arraybackslash}p{4.4cm}>{\raggedright\arraybackslash}p{4.9cm}>{\raggedright\arraybackslash}p{4.9cm}}
\toprule
Result & Reference & Scope \\ \midrule
Right-wing limit $\nu/(1-\beta)$ & Benaim--Friz--Lee \cite{BFL} & $\rho=0$ only \\
Moment explosions, smile slopes & Andersen--Piterbarg \cite{AP}; Keller-Ressel \cite{KR}; Lee \cite{Lee} & $\beta=1$ regime; affine models; model-free \\
Tail-to-smile transfer & Benaim--Friz \cite{BF,BF08}; Gulisashvili \cite{Gul10,Gul12}; Gao--Lee \cite{GL}; Caravenna--Corbetta \cite{CC}; Lucic \cite{Lucic} & model-free; includes the regime where all moments are finite \\
Sharp correlated right tail & Forde--Zhang \cite{FZ} & $\beta=1$ only; moment explosion; discontinuity at $\rho=0^-$ \\
Mass at zero and the left wing & Gulisashvili--Horvath--Jacquier \cite{GHJmass}; De~Marco--Hillairet--Jacquier \cite{DMHJ}; Gulisashvili \cite{Gatoms} & uncorrelated SABR; model-free atom asymptotics; small-strike side \\
Boundary hitting, SABR plane & Gulisashvili--Horvath--Jacquier \cite{GHJhit} & auxiliary geometric process \\
Survival probability & Yang--Wan \cite{YW}; Chen--Yang \cite{CY} & perturbative expansion; small-time regime \\
Correlated$\to$uncorrelated map & Antonov--Spector \cite{AS}; Antonov--Konikov--Spector \cite{AKS13,AKS19} & short-time-matched approximation \\
Geodesic expansions & Paulot \cite{Paulot} & small time; explicit boundary caveat \\
Regularity classification & D\"oring--Horvath--Teichmann \cite{DHT} & symmetric Dirichlet forms only if $\rho=0$ or $\beta\in\{0,1\}$ \\
Large maturity & Forde--Pogudin \cite{FP} & $(\beta=1,\rho\le0)$ or $(\beta<1,\rho=0)$ \\
\bottomrule
\end{tabular}
\caption{Selected results on extreme strikes and boundary effects for SABR-type models. None
treats the fixed-maturity extreme right wing of correlated SABR with $\beta\in(0,1)$.}\label{tab:lit}
\end{table}

At $\beta=1$ there is no absorption. Forde and Zhang \cite{FZ} show that the right tail
for $\rho<0$ is governed by the critical moment $m^*=(1-\rho^2)^{-1}$ \cite{LM}, with a
one-sided discontinuity at $\rho=0^-$; the implied-volatility smile grows rather than
approaching a finite right-wing limit, consistently with Lee's moment formula
\cite{Lee}. For $0<\beta<1$, by contrast, moments of every positive order are finite
(\cref{rem:martingale}), and the right-wing limit is finite; its value is $C^1$ but not $C^2$ at $\rho=0$ (\cref{cor:resolution}).
The same factor $(1-\rho^2)^{-1}$ thus appears in the two regimes through distinct
mechanisms: moment explosion at $\beta=1$, absorption at $0<\beta<1$.

The normal SABR model $\beta=0$ supplies the complementary boundary case. Its state space
is all of $\mathbb{R}$ and it has no absorbing boundary; its right-wing limit is $\nu$ for
every $\rho$ \cite[Example~10.4]{HL}. The limit $\beta\downarrow0$ in
\cref{cor:wing}, however, retains the factor $\sqrt{1-\rho^2}$, because it
approaches $\beta=0$ through models that still have an absorbing boundary at $F=0$. The
difference is caused by the change of state space, not by the CEV exponent alone.

Finally, the mass-at-zero literature \cite{GHJmass,DMHJ,Gatoms} describes how the atom
at the origin shapes the left wing. \cref{thm:main-tail} and
\cref{cor:wing} show that for negative correlation the same boundary also
changes the right-wing limit, acting through the survival constraint rather than through
the size of the atom. Taken together, these regimes show that the diffusion metric alone
does not determine the extreme wing: its interaction with the survival domain is
decisive.

\section{Conclusion}\label{sec:conclusion}
To leading order, the right wing of the absorbed SABR smile is now determined for every
$\beta\in(0,1)$ and every correlation: its limit is $\nu s_\rho/(1-\beta)$
(\cref{cor:wing}). Three points seem new beyond the formulas. First, the absorbing
boundary acts on the right wing through survival, not through the mass at zero: for
$\rho<0$ the noise orthogonal to volatility must stay above a moving square-root barrier,
and paying for this raises the cost of the volatility excursion by the factor
$(1-\rho^2)^{-1}$, as the energy identity \eqref{eq:energy-identity} shows. Second, the proofs need no large-deviation
principle: the upper bound reduces to a deterministic energy inequality whose key step is
that raising the orthogonal component early only costs more (\cref{lem:defer}), and the
lower bound to an Ornstein--Uhlenbeck process held above a constant level. Third, the tail
exponent is $d_D^2/(2T)$, with $d_D$ the intrinsic distance of the survival domain
(\cref{cor:corrected-c2}); this is why the estimate of \cref{conj:HL2}, built on the
unrestricted distance, does not hold for $\rho<0$ although its hypothesis does
(\cref{cor:conj2}). In practice, the leading term of Hagan's formula has the conjectured
limit for every $\rho$, so for negative correlation the formula overstates the right wing,
visibly already at moderate strikes (\cref{sec:num}).

The results are of leading order. Open are a remainder, or a rate of convergence for the
implied volatility, including the $k^{3/2}$ correction that the formal analysis of
\cref{sec:fiber} predicts (\cref{conj:subleading}); a heat-kernel proof of
\cref{cor:corrected-c2}, which needs Gaussian bounds for a potential and a connection that
are singular at the boundary (\cref{rem:heat-kernel}); and the extension to neighbouring
models such as absorbed normal SABR, shifted SABR and mean-reverting volatility.

\appendix
\crefalias{section}{appendix}

\section{Measurability, Versions and Uniqueness}\label{app:meas}
This appendix records the measurability and well-posedness facts used when conditioning
on a volatility path, passing to operational time, and summing level-hitting
configurations; it introduces no additional probabilistic estimate.

(i) On every compact subset of $(0,\infty)\times(0,\infty)$ the coefficients of the
$(F,a)$ equation are locally Lipschitz. Standard localization therefore gives a
pathwise-unique strong solution up to $\zeta$, and the model convention fixes its
continuation by $F_t=0$ after $\zeta$. We use continuous adapted versions throughout. The
stopped It\^o identities at levels $X\ge r$, $r\in\mathbb Q_{>0}$, hold simultaneously
outside one null set; continuity then gives the all-times formulation of
\cref{lem:cev}.

(ii) The volatility path space $C([0,T])$ is Polish. For a fixed continuous positive path
$y$, the It\^o integral $M(y)=\int_0^\cdot y_s\,dZ_s^\perp$ with deterministic integrand
is a centred continuous Gaussian process with covariance $V_{s\wedge t}(y)$. Since $a$ is
$\sigma(Z)$-measurable, independence of $Z^\perp$ and $Z$ implies independence of
$Z^\perp$ and $a$, and the deterministic-integrand stochastic-integral theorem gives, for
$\Prob_a$-almost every $y$, the conditional path law
\[
\mathcal L(M\mid a=y)=\mathsf W\circ\Pi_y^{-1},
\qquad \Pi_y(w)(t)=w(V_t(y)),
\]
as stated in \cref{lem:cond} \cite[Ch.~3]{KS}. The map $(y,w)\mapsto w\circ V(y)$ is
Borel measurable for the compact-open topologies, so
$\mathsf K(y,A):=\mathsf W(\Pi_y^{-1}(A))$ is a measurable probability kernel. Regular
conditional distributions are unique only $\Prob_a$-almost everywhere, and every
conditional assertion in the paper is made with that qualification.

(iii) Since $y>0$ for $\Prob_a$-almost every path, $V(y)$ is $C^1$ and strictly
increasing; its inverse $t_y:[0,V_T(y)]\to[0,T]$ is $C^1$, and under the kernel above
$(M_{t_y(v)})_{0\le v\le V_T(y)}$ has Wiener law on the finite random interval
$[0,V_T(y)]$. In particular, the comparison statements of \cref{sec:pos-lower} are
confined to this operational-time horizon.

(iv) The level-hitting times $\tau_j$ are stopping times of the continuous adapted
process $a$, and all skeleton data ($v_j$, $h_j$, $H$ and bin memberships) are
$\sigma(a)$-measurable. After coincident operational-time nodes are merged, the QP in
\eqref{eq:QP} is the minimum of a finite-dimensional coercive convex functional over a
closed polyhedron and is a Borel function of its operational-time and threshold data.
The conditional applications of \cref{lem:A2} and the tower property are therefore
legitimate.

\section{Proof of the Transfer Lemma}\label{app:transfer}
This appendix proves \cref{lem:transfer}. The argument is model-free and uses only
$\E F_T<\infty$.

\begin{proof}[Proof of \cref{lem:transfer}]
\emph{(i) Prices.} By the layer-cake identity, $\mathcal C(k)=f_0\int_k^\infty e^{x}P(x)\,dx$
(Fubini, using $\E F_T<\infty$). Fix $\varepsilon>0$ and
$k_0$ with $e^{-(\Lambda+\varepsilon)x^2}\le P(x)\le e^{-(\Lambda-\varepsilon)x^2}$
for $x\ge k_0$. Substituting $x=k+u$ and using $(k+u)^2\ge k^2+2ku$, for large
$k$
\begin{equation*}
\begin{aligned}
\mathcal C(k)&\le f_0e^{k-(\Lambda-\varepsilon)k^2}\int_0^\infty e^{u(1-2(\Lambda-\varepsilon)k)}\,du
= \frac{f_0\,e^{k-(\Lambda-\varepsilon)k^2}}{2(\Lambda-\varepsilon)k-1},\\
\mathcal C(k)&\ge f_0\int_k^{k+1}e^{x}P(x)\,dx\ge f_0e^{k-(\Lambda+\varepsilon)(k+1)^2}.
\end{aligned}
\end{equation*}
Hence $-\ln\mathcal C(k)/k^2\to\Lambda$. Since $0<\mathcal C(k)<f_0$, the implied
volatility $\sigma_{\BS}(k)\in(0,\infty)$ exists and is unique by strict monotonicity of
the Black--Scholes price in $\sigma$.

\emph{(ii) Black--Scholes tail.} With $d_1=(-k+\sigma^2T/2)/(\sigma\sqrt T)$ and
$d_2=d_1-\sigma\sqrt T$ one has
\begin{equation*}
d_1^2-d_2^2=-2k,\qquad\text{hence}\qquad e^{k}\varphi(d_2)=\varphi(d_1),
\end{equation*}
and the Mills expansion $\Phi(-x)=\varphi(x)(x^{-1}-x^{-3}+O(x^{-5}))$ as $x\to\infty$
gives
\[
C_{\BS}(k,\sigma)=f_0\bigl[\Phi(d_1)-e^k\Phi(d_2)\bigr]
= f_0\,\varphi(d_1)\,\frac{\sigma\sqrt T}{|d_1||d_2|}\bigl(1+O(k^{-2})\bigr),
\]
whence, uniformly on compact $\sigma$-sets,
\begin{equation}\label{eq:bs-tail}
-\ln C_{\BS}(k,\sigma)=\frac{k^2}{2\sigma^2T}+O(k),
\end{equation}
since $d_1^2/2=k^2/(2\sigma^2T)-k/2+\sigma^2T/8$ and $|d_1d_2|=k^2(\sigma^2T)^{-1}(1+O(1/k))$.

\emph{(iii) Identification of the limit.} Suppose $\sigma_{\BS}(k_j)\to\sigma^*\in(0,\infty]$
along a subsequence with $(\sigma^*)^2>1/(2\Lambda T)$. Pick
$\sigma_2\in(1/\sqrt{2\Lambda T},\sigma^*)$; eventually
$\mathcal C(k_j)\ge C_{\BS}(k_j,\sigma_2)\ge e^{-k_j^2(1+o(1))/(2\sigma_2^2T)}$,
contradicting $-\ln\mathcal C/k^2\to\Lambda>1/(2\sigma_2^2T)$. The case
$(\sigma^*)^2<1/(2\Lambda T)$ (including $\sigma^*=0$) is symmetric. Hence
$\sigma_{\BS}(k)\to1/\sqrt{2\Lambda T}$.
\end{proof}

\section{Numerical Methods}\label{app:num}

This appendix describes how the values of \cref{sec:num} are computed; one script, with
recorded seeds, regenerates all of them (see the Supplementary Material section).

\emph{Monte Carlo estimator.} We write $\mathcal C(k)=\E[c(a)]$, where $c(a)$ is the call
price conditional on the volatility path; conditional simulation given the volatility path
is standard for SABR \cite{COW}. Volatility paths are sampled exactly on a grid of
$500$ steps, from a mixture of exponentially tilted path measures, and reweighted
by the mixture likelihood ratio, which keeps the estimator unbiased whatever the tilts.
The tilts are built from the optimal log-volatility path of the deterministic problem
below. Conditionally on the volatility path, $X$ is simulated by an Euler scheme in
operational time based on \cref{lem:cev}, with $128$ antithetic inner draws, an
exponential tilt of $B$ toward the terminal target, and the Brownian-bridge correction
\begin{equation*}
p_{\mathrm{hit}}=\exp\Bigl(-\frac{2X_sX_{s+1}}{\orth^2\Delta v}\Bigr)
\end{equation*}
for absorption between grid points. Each cell uses $60{,}000$ volatility paths. Implied
volatilities are obtained by inverting the Black--Scholes formula in high-precision
arithmetic, and the reported errors are the call-price standard errors mapped through this
inversion. A cell is shown only if the effective sample size $(\sum_iW_i)^2/\sum_iW_i^2$ of
its call-price weights is at least $200$; this rule was fixed before the runs, and
every cell passed it.

\emph{Hagan's formula.} We use Hagan's expansion \cite{HKLW} with Ob{\l}{\'o}j's
leading term \cite{Obloj}:
\begin{equation*}
\sigma_{\mathrm{HO}}(k)=\frac{\nu\ln(f_0/K)}{x(z)}\Bigl[1+\Bigl(\frac{(1-\beta)^2\alpha^2}{24(f_0K)^{1-\beta}}
+\frac{\rho\beta\nu\alpha}{4(f_0K)^{(1-\beta)/2}}+\frac{(2-3\rho^2)\nu^2}{24}\Bigr)T\Bigr],
\end{equation*}
with $z=\nu(f_0^{1-\beta}-K^{1-\beta})/(\alpha(1-\beta))$ and
$x(z)=\ln\bigl((\sqrt{1-2\rho z+z^2}+z-\rho)/(1-\rho)\bigr)$. As $K\to\infty$ the first two
terms in the bracket vanish and $\nu\ln(f_0/K)/x(z)\to\nu/(1-\beta)$, which gives the limit
stated after \cref{fig:wing}.

\emph{Variational problems.} With the volatility driven by a control $\varphi$,
$a(t)=\alpha e^{\nu\varphi(t)}$, operational time $V(t)=\int_0^ta(s)^2ds$ and correlation
term $\xi(t)=(\rho/\nu)(a(t)-\alpha)$, the survival problem is
\[
R(k)=\min_{\varphi(0)=0}\Bigl\{\frac12\int_0^T\dot\varphi(t)^2\,dt
+\min_{\psi}\frac1{2\orth^2}\int_0^{V(T)}\dot\psi(v)^2\,dv\Bigr\},
\]
where the inner minimum is over paths $\psi$ with $\psi(0)=0$,
$\psi(V(T))\ge m-X_0-\xi(T)$, and $\psi(V(t))\ge-X_0-\xi(t)$ for all $t\le T$; the
terminal problem, used for $\rho=0$, omits the last constraint. We take $\varphi$ piecewise
linear on $20$ control nodes. The inner problem is then solved exactly: its minimizer is
the least nondecreasing concave majorant of the constraint points, anchored at the
origin. The outer minimization uses L-BFGS-B followed by Nelder--Mead from several
initializations; the computation is deterministic.

\section*{Supplementary Material}
The supplementary material, available from the authors on request, contains the code
and data for \cref{sec:num}, in which one script with recorded seeds regenerates every
value of \cref{tab:hagan} and \cref{fig:hagan}; the scripts for the other figures; and
a Lean development. The Lean development formalizes the
fixed-volatility Black--Scholes call-price rate, the identification of the
implied-volatility limit from a call-price rate, and the intrinsic-distance calculation
of \cref{prop:killed-distance}, and derives \cref{cor:corrected-c2}
from that proposition with the two tail bounds as hypotheses. It states the stochastic
inputs of the main proof as explicit interfaces and is not a formal proof of
\cref{thm:main-tail}. A claim manifest records which statements are formalized and links
them to Lean declarations.

\section*{Acknowledgments and Tool Disclosure}
\emph{AI-use disclosure.} The authors used Anthropic Claude Code and OpenAI Codex as
interactive research and writing tools. These tools assisted with exploratory discussion,
testing and refinement of ideas, literature and source organization, code development and
verification, mathematical error checking, and editorial revision. The authors take full
responsibility for the content presented.

\printbibliography[heading=bibintoc]

\end{document}